\documentclass
{amsart}
\usepackage{amsmath, amsthm, amsfonts,amssymb}
\usepackage[margin=1in]{geometry}

\usepackage{mathrsfs}    \usepackage{pst-node}
\usepackage{tikz-cd}                  
\usepackage{mathtools}
\mathtoolsset{showonlyrefs}

\usepackage{booktabs}
\definecolor{darkblue}{rgb}{0.0, 0.0, 0.8}

\usepackage[colorlinks=true, allcolors=darkblue]{hyperref}
\usepackage{todonotes}
\usepackage[all]{xy}
\usepackage{chngcntr}                      

\newtheorem{thm}{Theorem}[section]

\newtheorem*{question*}{Open Question}

\newtheorem{prop}[thm]{Proposition}
\newtheorem{rem}[thm]{Remark}
\newtheorem*{prop*}{Proposition}
\newtheorem*{thm*}{Theorem}
\newtheorem{lem}[thm]{Lemma}
\newtheorem*{lem*}{Lemma}
\newtheorem{cor}[thm]{Corollary}
\newtheorem{defn}[thm]{Definition}
\newtheorem*{rem*}{Remark}
\newtheorem*{cor*}{Corollary}

\allowdisplaybreaks

\def\al{\alpha} 
\def\be{\beta} 
 
\def\de{\delta} 
\def\ep{\varepsilon} 
\def\ze{\zeta} 
\def\et{\eta}

\def\ka{\kappa} 
\def\la{\lambda} 
\def\rh{\rho} 
\def\si{\sigma} 
 
\def\ph{\varphi} 
 
\def\ps{\psi} 
\def\om{\omega} 
\def\Ga{\Gamma}

\def\La{\Lambda}

\def\Om{\Omega}
\def\const{\Lambda}
 
\def\o{\circ} 
\def\x{\times}
\def\p{\partial} 
\def\X{{\mathfrak X}}
  
\def\L{\mathcal{L}}

\def\R{{\mathbb R}}

\def\coloneqq{:=}

\let\on=\operatorname

\def\Lag{\operatorname{Lag}}

\def\Met{\on{Met}}
\def\Diff{\on{Diff}}
\def\Prob{\on{Prob}}
\def\Dens{\on{Dens}}

\def\vol{\on{vol}}

\def\Id{\on{Id}}

\def\GDiff{G^{W,\on{Diff}}}
\def\GOrb{G^{W,\on{Orb}}}
\def\GProb{G^{W,\Prob}}
\def\Orb{\on{Orb}(g_0)}

\long\def\comment#1{}

\begin{document}

\title[Unbalanced Metric Transport]
{
The Wasserstein--Ebin Metric: A Geometric Lift of Unbalanced Optimal Transport to the space of Riemannian metrics}
\author{Martin Bauer,  Peter W. Michor, Fran\c{c}ois-Xavier Vialard}
\address{
Martin Bauer: 
Department of Mathematics, Florida State University, USA
\\Peter W. Michor:
Fakult\"at f\"ur Mathematik, Universit\"at Wien,
Nordbergstrasse 15, A-1090 Wien, Austria.
\\
F.-X. Vialard: LIGM, Universit\'e Gustave Eiffel, Champs-sur-Marne, France.}

\email{mbauer2@fsu.edu}

\email{Peter.Michor@univie.ac.at}
\email{francois-xavier.vialard@univ-eiffel.fr}
\date{\today}

\begin{abstract}
We introduce dynamic and static formulations that formally extend unbalanced optimal transport from the space of positive densities to the space of Riemannian metrics.

The first construction is based on a dynamic variational formulation in which the evolution of a Riemannian metric is driven by transport together with a source term. Choosing the $L^2$-metric to penalize the transport vector field and the Ebin metric to penalize the source component yields a new Riemannian metric on the manifold of Riemannian metrics, which we call the Wasserstein--Ebin metric.
Our main result shows that the volume map defines a Riemannian submersion from the Wasserstein--Ebin metric to the Wasserstein--Fisher--Rao metric on the space of smooth densities. In addition, we construct a Riemannian submersion from the automorphism group of the tangent bundle onto the space of Riemannian metrics, providing a generalization of Otto's geometric description for the Wasserstein metric to the setting of the Wasserstein--Ebin metric. 

To propose a static formulation of unbalanced optimal Riemannian metric transport, we introduce two Kullback--Leibler-type divergences on the space of Riemannian metrics: one inspired by matrix information geometry, and another related, through the volume map, to the classical Kullback--Leibler divergence on densities. Establishing a link between the static and dynamic formulations remains an open direction for future work.
\end{abstract}

\maketitle
\setcounter{tocdepth}{1}

\tableofcontents
\section{Introduction}
\subsection{Background and Motivation}Optimal transport provides a variational framework for comparing mass distributions through transport costs and has developed into a central tool in analysis, geometry, and applications, see e.g. the seminal textbooks of Villani~\cite{villani2009optimal,villani2003topics}. Motivated by increased demand from applications in data science, machine learning, and computer vision~\cite{peyre2019computational}, this classical theory has recently been extended in several different directions. This includes frameworks such as constrained optimal transport~\cite{benamou2015iterative}, the Gromov--Wasserstein distance~\cite{memoli2011gromov}, and vector-valued optimal transport~\cite{chen2018vector}. Most relevant for the present work is the setting of unbalanced optimal transport~\cite{chizat2018interpolating,liero2018optimal,piccoli2014generalized}, and in particular the Wasserstein--Fisher--Rao metric, which extends optimal transport to situations where the distributions under consideration are allowed to have differing masses.

In this article, we extend the optimal transport framework in a different direction, namely to the space of Riemannian metrics on a smooth manifold. 
This space is a fundamental object in infinite-dimensional differential geometry and mathematical physics, appearing naturally whenever geometric structures are treated as dynamical variables rather than fixed background data. 
In general relativity, spacetime geometry is encoded by a metric tensor satisfying nonlinear evolution equations~\cite{dewitt1967canonical}, while in elasticity theory changes of metric encode strain and provide a geometric description of deformation; see~\cite{marsden1994mathematical}.  More recently, the space of Riemannian metrics has also become relevant in imaging and shape analysis, where metric tensors are used to encode anisotropic structural information, for instance in diffusion tensor imaging~\cite{pennec2006riemannian} or for encoding geometric structures~\cite{zhang2021elastic}, and also in mechanics \cite{Kolev2024}.

The importance of the space of Riemannian metrics has led to the introduction of several weak Riemannian structures, each designed to capture different geometric and analytic aspects of this infinite-dimensional manifold. Examples include the Ebin $L^2$-metric~\cite{ebin1967space,ebin1970manifold,dewitt1967canonical}, volume-weighted versions of it~\cite{clarke2013conformal}, and higher-order Sobolev metrics~\cite{bauer2013sobolev,bauer2022smooth}. 

The purpose of the present article is to lift unbalanced optimal transport from the space of densities to the space of Riemannian metrics.  
Unbalanced optimal transport enjoys many equivalent formulations, in particular the dynamic and the static formulations \cite{chizat2018unbalanced,liero2018entropy}. Taking inspiration from this field, we propose in this paper several formulations, both dynamic and static, for the corresponding variational problems on the space of Riemannian metrics. Unlike existing matrix-valued transport models~\cite{ChenGeorgiouTannenbaum2018MatrixOMT,brenier2019monge}, where positive semidefinite matrices are not transported in a natural geometric manner and rather treated as independent variables, our model treats the Riemannian metric itself as the transported geometric object, with transport induced by the diffeomorphism action.

\subsection{Contributions and Open Questions}
An obvious difficulty in defining a transport-based framework in the context of the space of Riemannian metrics is that the action of the diffeomorphism group on the space of Riemannian metrics is not transitive. As a first step, we therefore study a balanced transport model restricted to individual diffeomorphism orbits. This leads to a Wasserstein-type metric on each orbit and reveals a simple geometric relation with classical optimal transport: the volume map from a diffeomorphism orbit of metrics to the space of smooth probability densities is a Riemannian submersion onto Wasserstein space, cf. Theorem~\ref{thm:riemsub_W}.

Since transport generated by diffeomorphisms alone remains confined to individual orbits, extending the construction to the full manifold of Riemannian metrics requires an additional mechanism that accounts for metric variations transverse to the diffeomorphism action. We achieve this by introducing an unbalanced formulation in which the source term is penalized by the Ebin metric. Combining this with the transport contribution induced by the \(L^2\)-metric on vector fields yields a new Riemannian metric on the space of Riemannian metrics, which we call the Wasserstein--Ebin metric. Its definition is given by an infimal convolution of the transport and source contributions and may be viewed as the geometric analogue of the Wasserstein--Fisher--Rao metric.

Our first main result, cf. Theorem~\ref{thm:riemsub_WFR}, shows that the volume map defines a Riemannian submersion from the Wasserstein--Ebin metric to the Wasserstein--Fisher--Rao metric on the space of smooth densities. Furthermore, in Theorem~\ref{thm:Otto_WE}, we construct a Riemannian submersion from the automorphism group of the tangent bundle onto the space of Riemannian metrics, extending Otto's geometric description of the Wasserstein metric to the Wasserstein--Ebin setting. A complete overview of the various Riemannian submersion constructions is then presented in Theorem~\ref{thm:diagram}.

An important analytical issue concerns the geodesic distance induced by the Wasserstein--Ebin metric: as shown by Eliashberg-Polterovich~\cite{eliashberg1993bi} and Michor-Mumford~\cite{michor2005vanishing} in an infinite-dimensional setting the geodesic distance of a Riemannian metric can be degenerate or even vanish identically, i.e., there can exist distinct points such that their geodesic distance is zero; see also~\cite{jerrard2019vanishing,magnani2020remark,bauer2020vanishing} for further examples and discussions of this phenomenon. The Wasserstein--Ebin metric, studied in this article, can be interpreted as the infimal convolution of two weak Riemannian metrics with non-degenerate geodesic distance functions. While one might expect this to guarantee non-degeneracy of the induced geodesic distance, this is in general not the case; inspired by a construction of Magnani and Tiberio in~\cite{magnani2020remark}, we show in Theorem~\ref{thm:vanishing} that infimal convolution of weak Riemannian metrics may even lead to vanishing of the geodesic distance, even when both underlying Riemannian metrics individually have non-degenerate geodesic distance. For this reason, the distance problem for the Wasserstein--Ebin metric cannot be reduced directly to known properties of either the Ebin metric or the $L^2$-metric on the diffeomorphism group.

Our general discussion of infimal convolution, as presented in Section~\ref{sec:infimal}, therefore serves two purposes: it provides the abstract framework in which the Wasserstein--Ebin metric naturally fits, and it explains why standard arguments for non-degeneracy are insufficient in the present context. For the specific case of the Wasserstein--Ebin metric we prove the non-degeneracy for Riemannian metrics with distinct volume densities and for its restriction onto a fixed orbit of the diffeomorphism group, cf. 
Lemmas~\ref{lem:geodesic_dist} and \ref{lem:geod_orbit}. The first one is achieved by obtaining an explicit 
lower bound in terms of the Wasserstein--Fisher--Rao distance between the induced volume densities. These estimates, however, do not exclude the possibility of collapse for arbitrary Riemannian metrics. Although the results developed in this article lead us to conjecture that the geodesic distance should remain non-degenerate, at the current time we have to leave this question open for future investigations. 

Finally, in Section~\ref{sec:static} we propose a formal static formulation of unbalanced Riemannian metric transport. This leads naturally to the introduction of Kullback--Leibler-type divergences on the space of Riemannian metrics, which we expect to be of independent interest on its own.  Specifically, we introduce two variants: the first is inspired by matrix information geometry~\cite{NielsenBhatia2013}, while the second projects, via the volume map, onto the classical Kullback--Leibler divergence on the space of densities. We show that both define bona fide divergence functions on $\mathrm{Met}(M)$ and that their second variations recover the Ebin $L^2$-metric. Establishing a link (e.g. an inequality) between the corresponding static formulation and the dynamic formulation of unbalanced Riemannian metric transport remains, however, an open problem. We conclude the paper with Section~\ref{sec:Euleralpha} by describing a potential connection to the Euler--$\alpha$ equation, mirroring the situation for unbalanced optimal transport and the Camassa-Holm equation~\cite{GALLOUET20184199}.

\subsection{Relation to other Matrix-Valued Transport Models}
It is natural to compare the present construction with models for transport of matrix-valued densities. Two different paths were followed: extension of the static formulation of optimal transport such as in \cite{chen2020matrix}, or extension of the dynamic formulation as done in~\cite{brenier2019monge}. In general, Kantorovich--Bures type geometries~\cite{bhatia2019bures} are used in these models. The closest model to ours rely on the use of the Benamou--Brenier dynamic formulation, extended to the SPD matrices while retaining convexity such as in \cite{ChenGeorgiouTannenbaum2018MatrixOMT}. A related geometric viewpoint has recently been developed by Khesin and Modin~\cite{khesin2025universal}, where vector- and matrix-valued transport problems are formulated through Riemannian submersions associated with enlarged symmetry groups extending the diffeomorphism action. The present framework differs from these matrix-valued transport models in two essential ways. First, the transported object is a Riemannian metric itself, so that transport is induced intrinsically by the diffeomorphism action on the manifold of metrics rather than by independent fiberwise interpolation of positive semidefinite matrices. Second, a Riemannian metric canonically determines a volume density through its determinant, and this distinguished scalar quantity gives rise, via the volume map, to the Wasserstein--Fisher--Rao geometry. In summary, the Wasserstein--Ebin metric should be viewed as a transport geometry on the space of Riemannian metrics rather than as a matrix-valued transport model in the usual sense.

\subsection{Structure of the Article}
The article is organized as follows. In Section~\ref{Sec:background} we review the geometric formulation of balanced and unbalanced optimal transport and fix notation for the relevant infinite-dimensional manifolds. Section~\ref{sec:infimal} discusses infimal convolution of Riemannian metrics in an abstract setting and includes a general observation on possible degeneracy of the induced geodesic distance. In Section~\ref{sec:ORMT} we introduce a balanced transport model on diffeomorphism orbits of Riemannian metrics and analyze its relation to classical Wasserstein geometry. Finally, in Section~\ref{sec:OURMT} we present the main contributions of the article: the construction of the Wasserstein--Ebin metric together with the corresponding Riemannian submersion results and the tangent bundle automorphism picture.

\subsection{Acknowledgements and Data Availability Statement}
The authors are grateful to Boris Khesin and Klas Modin for valuable comments and discussions during the preparation of this document. M. Bauer was partially supported by NSF grant DMS-2526630 and by the Binational Science Foundation (BSF). F.-X. Vialard was partially supported
by the B\'ezout Labex (New Monge Problems), funded by ANR, reference ANR-10-LABX-58.

Data sharing is not applicable to this article as no datasets were generated or analyzed during the current study.

\section{Background}\label{Sec:background}
In this section we will introduce the basic notation, which we will use throughout the article and review the dynamic (Riemannian) view point on balanced and unbalanced optimal transport. 
\subsection*{Manifolds of Mappings}\label{sec:manifolds_mappings}
We start with introducing several infinite dimensional manifolds of mappings, which will play a central role in the remainder of the article. We will begin with introducing those spaces for a compact parameter space $M$, i.e., 
let $M$ be a closed manifold of dimension $d\geq 1$. We then consider  the group of all smooth and orientation preserving diffeomorphisms, which we denote by $\Diff(M)$. Furthermore, we denote by $\Dens(M)$ the space of smooth positive densities and by $\Prob(M)$ the subspace of smooth probability densities, i.e., 
\begin{align*}
\Dens(M)&:=\{\mu\in \Ga(|\La^d|(T^*M)) : \mu>0\}\\
\Prob(M)&:=\{\mu\in \Dens : \int\mu=1\},
\end{align*}
where $|\La^n|(T^*M)$ is the line bundle of densities on $M$, and the positivity condition means the positivity of the Radon--Nikodym derivative $\frac{\mu}{\lambda}$ of $\mu$ with respect to some volume form $\lambda$ on $M$. Finally, we need to introduce the central object of the present article, the space $\Met(M)$ of all smooth Riemannian metrics on $M$:
\begin{align*}
\Met(M):=\left\{
g \in \Gamma^\infty\!\bigl(\operatorname{Sym}^2 T^*M\bigr)
\;\middle|\;
g_x \text{ is positive definite for every } x \in M
\right\},
\end{align*}
where $\operatorname{Sym}^2 T^*M$ denotes the bundle of symmetric covariant 2-tensors on the manifold $M$.

In parts of the article, we will also consider the non-compact manifold $M=\mathbb R^d$. In this case, we have to proceed with some care to define the corresponding manifolds of mappings with certain boundary conditions. Namely we first introduce the space 
$H^\infty(\mathbb R^d)=\bigcap_{k\ge 1}H^k(\mathbb R^d)$, which is the intersection of all Sobolev 
spaces, and thus is a reflexive Fr\'echet space. This allows us to define the corresponding mapping spaces on $\mathbb R^d$. Namely we have:
\begin{align*}
 &\Diff(\mathbb R^d)=\bigl\{\ph=\Id+f: f\in H^\infty(\mathbb R^d)^n, 
		\det(\mathbb I_n + df)>0\bigr\}\,,\\
        &\Dens(\mathbb R^d):=\{\mu=(1+f)\lambda: f\in H^{\infty}(\mathbb R^d)\text{ and }f>-1\}\\
        &\Prob(\mathbb R^d):=\{\mu\in \Dens(\mathbb R^d): \int_{\mathbb R^d} \mu=1\}\\
        &\Met(\mathbb R^d) = \{g = g_0+h: h 
\in  H^{\infty}(\mathbb R^d,\operatorname{Sym}^2 \mathbb R^d), g_x \text{ is positive definite for every } x \in \mathbb R^d\},
\end{align*}
where $g_0$ denotes the Euclidean metric on $\mathbb R^d$ and where $\lambda$ denotes the Lebesgue measure. 

Analogous spaces exist also for non-compact Riemannian manifolds $(M,g_0)$; see \cite{Michor20}.

\subsection*{Dynamic Unbalanced Optimal Mass Transport and Otto's Riemannian Submersion}
In this section, we will briefly recall the dynamic formulation of (unbalanced) optimal transport and describe Otto's Riemannian submersion construction both in the balanced~\cite{OttoPic} and unbalanced situation~\cite{GALLOUET20184199}. Furthermore, we will restrict ourselves solely to the smooth category and ignore any subtleties that arise from considering low regularity measures (mappings, resp.). 
\begin{defn}[Dynamic Optimal Mass Transport]
Let $(M,g_0)$ be a Riemannian manifold. Given smooth probability densities $\rho_0, \rho_1\in \operatorname{Prob}(M)$ the dynamic optimal mass transport problem consists of minimizing the Lagrangian
\begin{equation}\label{eq:standardOMT}
\Lag(\rho) =  \int_0^1 \| v(t,\cdot)\|_{g_0, \rho}^2 dt
\end{equation} 
over all paths $\rho:[0,1]\to \operatorname{Prob}(M)$ with $\rho(0)=\rho_0$ and $\rho(1)=\rho_1$
subject to the continuity equation 
\begin{equation}\label{eq:cont_original}
\dot \rho(t)= -\mathcal L_{v(t)} \rho(t)
\end{equation}
Here $\|\cdot\|_{g_0,\vol(g)}$ denotes a norm on the space of vector fields $\mathfrak X(M)$--depending on the source Riemannian metric $g_0$ and on the  volume density $\rho(t)$. 

The \emph{Wasserstein-OMT} problem corresponds to~\eqref{eq:standardOMT} with 
\begin{equation}\label{eq:l2metric}
\|v\|^2_{g_0,\rho} = \int_M g_0(v,v) \rho
\end{equation}
being the $L^2$-product with respect to $\rho$.  
\end{defn}
By the celebrated result of Otto~\cite{OttoPic} this construction 
indeed lends a geometric view point for the Wasserstein distance on the space of smooth probability densities, namely the mapping
\begin{align} \pi: \Diff(M) & \to \operatorname{Prob}(M) 
\\ \varphi & \mapsto \varphi_*(\rho_0)
\end{align} 
is a (formal) Riemannian submersion of the (non-invariant) $L^2$ metric on $\Diff(M)$ to the Wasserstein metric on the space of smooth probability densities.  

Next we describe an analogous picture in the case of unbalanced optimal mass transport (UOMT), as introduced in~\cite{GALLOUET20184199}. We start by reviewing again the dynamic formulation of this model:

\begin{defn}[Dynamic Unbalanced Optimal Mass Transport]
Let $(M,g_0)$ be a Riemannian manifold. Given smooth densities $\rho_0, \rho_1\in \operatorname{Dens}(M)$ the dynamic unbalanced optimal mass transport problem consists of minimizing the Lagrangian
\begin{align}\label{eq:standardUOMT}
\Lag(\rho) =  \int_0^1 \| v(t,\cdot)\|_{g_0, \rho}^2 dt + \const\int_0^1 G_{\rho}(f\rho,f\rho) dt
\end{align} 
over all paths $\rho:[0,1]\to \operatorname{Dens}(M)$ with $\rho(0)=\rho_0$ and $\rho(1)=\rho_1$
subject to the continuity equation 
\begin{equation}
\dot \rho(t)= -\mathcal L_{v(t)} \rho(t)+f\rho
\end{equation}
Here $\|\cdot\|_{g_0,\rho}$ denotes a norm on the space of vector fields $\mathfrak X(M)$ -- depending on the source Riemannian metric $g_0$ and on the  volume density $\rho(t)$ -- and 
$G$ denotes a Riemannian metric on the space of smooth, positive densities.

The \emph{Wasserstein--Fisher--Rao UOMT} problem corresponds to~\eqref{eq:standardUOMT} using again the  $L^2$-product with respect to $\rho$ for the norm on the velocity $v$ and by choosing $G$ to be the Fisher-Rao metric, i.e.,
\begin{equation}
G_{\rho}(f\rho,f\rho)=\int_M \frac{f\rho}{\rho}\frac{f\rho}{\rho}\rho=\int_M f^2\rh.
\end{equation} 
\end{defn}
Recently, Gallou\"{e}t and Vialard~\cite{GALLOUET20184199} showed that there exists a analogue of Otto's Riemannian submersion construction in this unbalanced setting. Therefore they considered the automorphism group
$\operatorname{Aut}(\mathcal{C}(M))$, where $\mathcal{C}(M)$ denotes the cone over $M$, i.e.,
 the Riemannian manifold $\mathcal{C}(M) \coloneqq M \times \R_{>0}$ endowed with the cone metric $r^2g_0 + dr^2$. For later purposes we note that $\R_{>0}$ is a commutative group. The automorphism group of the cone $\operatorname{Aut}(\mathcal{C}(M))$ can be viewed as a semi-direct product via $\operatorname{Aut}(\mathcal{C}(M)) = \operatorname{Diff}(M) \ltimes \operatorname{Gau}(\mathcal{C}(M))$. This allows us  to consider the 
 Riemannian $G^{\on{Aut}(\mathcal{C}(M)}$ metric on $\operatorname{Aut}(\mathcal{C}(M))$ via 
\begin{equation}\label{eq:L2cone}
    G^{\on{Aut}(\mathcal{C}(M)}_{\varphi,\lambda}((\delta \varphi,\delta \lambda),(\delta \varphi,\delta \lambda))\coloneqq \int_M \lambda^2 g_0(\delta \varphi,\delta \varphi) + 4\const\left(\delta \lambda\right)^2 \operatorname{\vol}(g_0)\,.
\end{equation}
The analogue of Otto's Riemannian submersion is given by the following result, which is due to~\cite{GALLOUET20184199}:
\begin{thm}[Otto's Riemannian submersion for UOMT~\cite{GALLOUET20184199}]\label{thm:OttoUOT}
The map
\begin{align}\label{eq:pi0} \pi_0: \begin{cases}
\operatorname{Aut}(\mathcal{C}(M)) & \to \operatorname{Dens}(M) \\ (\varphi,\lambda) & \mapsto \varphi_*(\lambda^2\rho_0),
\end{cases}
\end{align}  
is a Riemannian submersion, where $\operatorname{Aut}(\mathcal{C}(M))$ is equipped with the $L^2$ metric~\eqref{eq:L2cone} and where $\operatorname{Dens}(M)$  carries the Wasserstein--Fisher--Rao metric
 \begin{equation}\label{eq:WFRmetric}
 \begin{aligned}&G^{\on{WFR}}_\rho(\delta \rho,\delta \rho)=\inf_{v,f}\int_M  g_0(v,v) \rho +{\const}\int_M f^2\rho \\
    &\qquad\qquad\text{subject to: }\quad \delta \rho= -\mathcal L_{v} \rho + f\rho . 
\end{aligned}
\end{equation}
\end{thm}
\subsection*{The Ebin Metric on the Space of Riemannian Metrics}
The space of all Riemannian metrics carries a natural weak Riemannian metric, called the \emph{Ebin metric}, defined at each \(g \in \operatorname{Met}(M)\) by
\[
G^{\on{E}}_g(h,k)=\int_M \operatorname{tr}\!\bigl(g^{-1} h\, g^{-1} k\bigr)\, \vol(g),\qquad h,k \in T_g\operatorname{Met}(M)\cong \Gamma^\infty(\operatorname{Sym}^2 T^*M),
\]
where \(\vol(g)\) denotes the volume form induced by \(g\). This metric endows \(\operatorname{Met}(M)\) with the structure of an infinite-dimensional weak Riemannian manifold and has been analyzed in considerable detail. In particular, explicit formulas are available for its geodesics and sectional curvature~\cite{freed1989basic,gil1991riemannian}. The geodesic distance is known to be non-degenerate, but the corresponding metric space is not complete. Its metric completion has been described by Clarke and it has been shown to admit the structure of a \(\operatorname{CAT}(0)\) space~\cite{clarke2013completion,clarke2013geodesics}.

\section{Infimal Convolution of Riemannian Metrics}\label{sec:infimal}
In this Section we will review  the concept of infimal convolution of Riemannian metrics on abstract manifolds. The main result of this part, Theorem~\ref{thm:vanishing}, will demonstrate a surprising degeneracy of the geodesic distance, which can appear in this context. 
 In the second part of the section we study the specific situation of infimal convolution of two Riemannian metrics, where one is induced by a Riemannian metric on a group of transformation acting on the manifold.

In all of this section let $\mathcal M$ be a (possibly infinite dimensional) manifold.
We start by defining the infimal convolution of two Riemannian metrics in this setting:
\begin{defn}[Infimal convolution of Riemannian metrics]
Let $G^1$ and $G^2$ be two Riemannian metrics on $\mathcal M$. For $x\in \mathcal M$ and $\delta x\in T_x\mathcal M$ The infimal convolution $G=G^1*G^2$ is defined via
\begin{align}
G_x(\de x,\de x)=\inf \big\{ G^1_x(v,v)+G^2_x(w,w): v+w = \de x\in T_x\mathcal M\big\}.
\end{align}
\end{defn}
An immediate consequence of this definition is the fact that, for all $x \in \mathcal M$ we have that $G_x \preccurlyeq G^i_x$ (inequality between quadratic forms) for $i = 1,2$. Note, that this implies the same inequality for the induced geodesic distances. 
Furthermore, we note that the above definition can be easily adapted to the case that $G^1$ and/or $G^2$ are only Sub-Riemannian metrics. Our interest in this concept stems from the fact, that the Wasserstein--Fisher--Rao metric, as introduced in the previous section, can be viewed as infimal convolution of the Wasserstein and the Fisher-Rao metric. 

Next we will observe an interesting phenomenon for Riemannian metrics that are obtained via infimal convolution; namely the potential degeneracy of the geodesic distance function even if both $G^1$ and $G^2$ induce a non-degenerate distance function by themselves.  Vanishing (degeneracy) of the geodesic distance is a purely infinite dimensional phenomenon, which has been observed for a variety of natural Riemannian metrics on spaces of sequences and functions, cf.~\cite{magnani2020remark,jerrard2019geodesic,michor2005vanishing,bauer2020vanishing} and the references therein. In the following theorem we will construct an explicit example that demonstrates this phenomenon for the space $\mathcal M=\ell^2$ equipped with a metric that is obtained by infimal convolution. 
\begin{thm}\label{thm:vanishing}
Let $m\colon\mathbb N\to(0,\infty)$ be a fixed sequence such that $\lim_{i\to\infty}m_i=0$ and  let $f\colon\mathbb R_{>0}\to \mathbb R_{>0}$  be a smooth, bounded function with $\lim_{x\to\infty}f(x)=0$. On the space of $\ell^2$-sequences we consider the weak Riemannian metrics 
 \begin{align}\label{eq:metric:l2}
 G^1_x(u,u)
 \coloneqq  
 \langle u, u\rangle_{\ell^2_m} 
 \coloneqq
 \sum_i m_i u_i u_i,
 \qquad
G^2_x(u,u)
 \coloneqq  
 f(\|x\|^2_{\ell^2}) \langle u, u\rangle,
 \end{align}
 where $x\in \ell^2$, and
 $u\in T_x\ell^2=\ell^2$.
Let $G$ be the Riemannian metric defined  via infimal convolution of $G^1$ and $G^2$.

Then $G^1$ and $G^2$ have both a non-degenerate geodesic distance function, but the geodesic distance of $G$ vanishes identically on all of $\ell^2$.
\end{thm}
We note that, in the context of general cost functions on abstract sets, an analogous phenomenon--namely, degeneracy of the induced distance of a cost obtained by infimal convolution of two costs with non-degenerate distances--has recently been shown in~\cite{de2025infimal}.

\begin{proof}
To show that this family of metrics has vanishing geodesic distance, we follow the construction of Magnani and Tiberio in~\cite{magnani2020remark}. 
 Therefore, let $\{e_n\}_{n=1}^{\infty}$ denote the basis of $\ell^2$ given by
 $e_1=(1,0,0,\ldots)$, $e_2=(0,1,0,\ldots)$ and so on and let $x$ and $y$ be two arbitrary elements of $\ell^2$. We aim to construct paths  of arbitrary short length, that
 connect $x$ to $y$. Therefore let for $t \in [0,1]$
 \begin{align*}
 c_1(t)=  x+t m_n^{-1/4} e_n,&&
 c_2(t)= x+ m_n^{-1/4} e_n+t(y-x)),&&
 c_3(t)=  y+(1-t) m_n^{-1/4} e_n.
 \end{align*}
 The concatenation of these three linear paths leads to a path  that connects $x$ to $y$. We will show that the length of this path can be made arbitrary short by choosing $n$ sufficiently large. We have:
 \begin{align*}
 \dot c_1(t)=  m_n^{-1/4} e_n,&&
 \dot c_2(t)= (y-x),&&
 \dot c_3(t)=  -m_n^{-1/4} e_n,
 \end{align*}
We bound the length of $c_1$ with respect to the infimal convolution metric by choosing $v=\dot c_1$ and $w=0$. Then
 \begin{align*}
  \on{Len}(c_1) &\leq \int_0^1  \| m_n^{-1/4} e_n\|_{\ell^2_m}   dt   = 
   \| m_n^{1/4} e_n\|_{\ell^2}  = m_n^{1/4},
 \end{align*}
 where we used that $e_n$ has unit length in $\ell^2$. A similar calculation shows that $\on{Len}(c_3)\leq m_n^{1/4}$ and thus these two segments can be made arbitrarily small. To bound the length of $c_2$, we choose $v=0$ and $w=\dot c_2$ leading to
 \begin{align*}
  \on{Len}(c_2) &\leq  \int_0^1 \sqrt{f(\|c_2(t)\|^2_{\ell^2})} \|y-x\|_{\ell^2} dt  = \int_0^1 \sqrt{f(\|c_2(t)\|^2_{\ell^2})} dt  \|y-x\|_{\ell^2}
 \end{align*}
 To bound this term we
estimate the $\ell^2$ norm along the curve $c_2$ using the reverse triangle inequality: 
\begin{equation*}
\|c_2(t)\|^2_{\ell^2}\geq 
m_n^{-1/2} - 2 m_n^{-1/4} (\|x\|_{\ell^2}+\|y-x\|_{\ell^2})
=m_n^{-1/4}(m_n^{-1/4} - 2 (\|x\|_{\ell^2}+\|y-x\|_{\ell^2}))
\end{equation*}
From here it is easy to see that $\lim_{n\to\infty} \|c_2(t)\|^2_{\ell^2}=\infty$. Thus by choosing $n$ sufficiently large, the length of $c_2$ can be made arbitrary small as well, which concludes the proof of vanishing geodesic distance.   

It remains to show the non-degeneracy of the geodesic distance of $G^1$ and $G^2$. For $G^2$ this is immediate as it is conformal to the standard $\ell^2$-metric, and thus a strong Riemannian metric on the space $\ell^2$, which directly implies the non-degeneracy of the geodesic distance. Finally, to see the non-degeneracy of $G^1$, note that $(\ell^2, G^1)$ is a pre-Hilbert space such that the geodesic distance is the norm distance, since $G^1$ does not depend on the choice of the foot point $x$.
\end{proof}

Next, we show that this ill-behavior of the geodesic distance can be ruled out if one of the Riemannian metrics dominates the other one on metric balls:
\begin{lem}\label{lem:non-degenerate}
Let $\mathcal M$ be an infinite dimensional manifold equipped with two Riemannian metrics $G^1$ and $G^2$, such that the geodesic distance function of $G^1$ is non-degenerate. Assume in addition that $G^2$ uniformly dominates $G^1$ on $G^1$-metric balls, i.e., for any $G^1$-metric  ball $B_r(x_0)$ of radius $r$ centered at $x_0$, there exists a constant $C>0$ (depending only on $r$ and $x_0$) such that 
\begin{equation}
 C G^1_x(u,u)\preccurlyeq G^2_x(u,u),\qquad x\in B_r(x_0), u\in T_x\mathcal M.
\end{equation}
Then the geodesic distance of both $G^2$ and of the infimal convolution $G=G^1*G^2$ are non-degenerate. 
\end{lem}
\begin{proof}
Let $x\neq y\in\mathcal M$. Choose some $G^1$-metric ball $B_r(x)$ such that $y\in B_r(x)$ and let $\gamma: [0,1]\to\mathcal M$ be any path, that connects $x$ to $y$. Without loss we assume that $\gamma$ does not leave the ball $B_r(x)$. The non-degeneracy of $G^2$ is immediate. We proceed to prove the non-degeneracy of  
$G=G^1*G^2$. 

Therefore let $h$ and $k$ be a decomposition of $\dot\gamma$ that approximates the infimum in the definition of $G$ up to some $\ep>0$. We then calculate
\begin{align}
\mathcal L^G(\gamma)+ \ep &=\int_0^1 \sqrt{G_{\gamma}(\dot\gamma,\dot\gamma)} dt +\ep
\ge\int_0^1 \sqrt{G^1_{\gamma}(h,h) + G^2_{\gamma}(k,k)} dt 
\\&\geq \int_0^1 \sqrt{G^1_{\gamma}(h,h) + CG^1_{\gamma}(k,k)} dt
\\&\geq \operatorname{min}(1,\sqrt{C})\int_0^1  \sqrt{G^1_{\gamma}(h,h) + G^1_{\gamma}(k,k)}dt
\\&\geq \frac12 \operatorname{min}(1,\sqrt{C}) \int_0^1  \sqrt{G^1_{\gamma}(h+k,h+k)} dt\\&=\frac12 \operatorname{min}(1,\sqrt{C}) \int_0^1  \sqrt{G^1_{\gamma}(\dot\gamma,\dot\gamma)}dt,
\end{align}
which concludes the proof since the geodesic distance of $G^1$ is non-degenerate.
\end{proof}

\subsection*{A Transport-Inspired Situation}
In the optimal transport frameworks, that we are interested in this article, the (Sub-)Riemannian metric $G^1$ is induced by a Riemannian metric on a group of transformation acting on the manifold $\mathcal M$, which we will study in abstract setting next. 
Therefore let $\mathcal G$ be a (possibly infinite-dimensional) Lie group with  Lie algebra $\mathfrak g = T_e\mathcal G$.
We denote group multiplication  by $\mu:\mathcal G\x \mathcal G\to \mathcal G$ with $\mu(\ph,\ps)= \ph.\ps = \mu_\ph(\ps) = \mu^\ps(\ph)$ and assume that $\mathcal G$ acts smoothly on $\mathcal M$ from the left via 
$$\mathcal G\times \mathcal M\ni(\varphi,x)\mapsto \ell(\ph,x)=\varphi. x = \ell_\ph(x) = \ell^x(\ph) \in \mathcal M.$$
 For $X\in \mathfrak g$ we denote by $\ze_X\in \X(\mathcal M)$ the infinitesimal action vector field given by 
$\ze_X(x) = T_e(\ell^x).X = T_{(e,x)}(X,0_x)\in T_x\mathcal M$, and we denote by $R_X\in\X(\mathcal G)$ the right invariant vector field. It is well known that in this situation we have (see \cite[6.2]{Michor08}):
\begin{itemize}
\item $\ze: \mathfrak g \to \X(\mathcal M)$ is linear and satisfies $[\ze_X,\ze_Y] = - \ze_{[X,Y]}$.
\item $T_x\ell_\ph.\ze_X(x) = \ze_{\on{Ad}(\ph).X}(\ph.x)$
\item $R_X\x 0_{\mathcal M}\in\X(\mathcal G\x \mathcal M)$ is $\ell$-related to $\ze_X\in \X(\mathcal M)$.
\end{itemize}

\begin{prop}\label{prop:infimalconv} In the situation above let 
\begin{itemize}
\item $G^2$ be a Riemannian metric on $\mathcal M$ which is invariant for the $\mathcal G$-action $\ell$,
\item  $G^1:\mathcal M \to \Met(\mathcal G)$ be  $\mathcal G$-invariant smooth in the sense that  
$$G^1(\ell_{\ph}(x)) = (\mu^{\ph})^*G^1(x), \quad\text{ or}\quad (\mu^{\ph})_*G^1 \ell_{\ph}(x)) = G^1(x).$$
\end{itemize}
Then $\ell:\mathcal G\x \mathcal M \to \mathcal M$ is a Riemannian submersion from the almost product metric 
$$
(G^1\x_\ell G^2)_{(\ph,x)}\big((\de\ph,\de x),(\de\ph,\de x)\big) = G^1(x)_\ph(\de\ph,\de\ph) + G^2_x(\de x,\de x)
$$
on $\mathcal G\x \mathcal M$ to the quotient (infimal convolution) metric on $\mathcal M$ which for $\ell(\ph,x)=\bar x$ is given by
\begin{equation}\label{eq:quotientmetric}
 (G^1 * G^2)_{\bar x}(\de\bar x,\de \bar x)  
 = \inf_{(\de\ph,\de x)\in T_\ph\mathcal G\x T_x \mathcal M} G^1(x)_\ph(\de\ph,\de\ph) + G^2_x(\de x,\de x)
\end{equation}
subject to
\begin{equation*}
    \de \bar x = T_{(\ph,x)}\ell .(\de\ph,\de x) = \ze_X(\ph.x) + T_x\ell_\ph.\de x.
\end{equation*}
\end{prop}

\begin{proof}
We consider  $\bar x = \ell(\ph,x)=\ph.x$ and $T_\ph\mathcal G\ni \de\ph=T_e\mu^\ph.X$ for $X\in\mathfrak g$. Then
\begin{align*}
T_{(\ph,x)}\ell .(\de\ph, \de x) &= T_{(\ph,x)}\ell(R_X(\ph)\x 0_x)+ (0_\ph\x \de x)) = \ze_X(\ph.x) + T_x\ell_\ph.\de x
\end{align*}
and we have to check that the following pointwise quotient metric~\eqref{eq:quotientmetric} does not depend on the point $(\ph,x)$ in the fiber $\ell^{-1}(\bar x)$.
If $\ph.x = \ell(\ph,x)=\bar x = \ell(\ps,y) = \ps.y$ then $\ps^{-1}.\ph.x = y$ and we have isomorphisms
$T_\ph\mu^{\ph^{-1}\ps}:T_\ph\mathcal G\to T_\ps\mathcal G$ and $T_x\ell_{\ps^{-1}\ph}:T_x\mathcal M\to T_y\mathcal M$ so that 
\begin{align*}
&\bar G^{(\ps,y)}_{\bar x}(\de\bar x,\de\bar x) = \inf\big\{G^1(y)_\ps(\de\ps,\de\ps) + G^2_y(\de y,\de y): (\de\ps,\de y)\in T_\ps\mathcal G\x T_y \mathcal M
\\&\qquad\qquad\qquad\qquad\qquad\qquad\qquad\qquad\qquad
\text{ with }  \de \bar x = T_{(\ps,y)}\ell .(\de\ps,\de y) \big\} 
\\&
= \inf\big\{G^1(\ell_{\ps^{-1}.\ph}(x)_{\mu^{\ph^{-1}\ps}(\ph)}(T_\ph\mu^{\ph^{-1}\ps}.\de\ph,T_\ph\mu^{\ph^{-1}\ps}.\de\ph) +
\\&\qquad
+ G^2_{\ell_{\ps^{-1}\ph}x}(T_x\ell_{\ps^{-1}\ph}.\de x,T_x\ell_{\ps^{-1}\ph}.\de x): (\de\ph,\de x)\in T_\ph\mathcal G\x T_x \mathcal M
\\&\qquad\qquad\qquad\qquad\qquad\qquad
\text{ with }  \de \bar x=T_{(\ps,y)}\ell .(T_\ph\mu^{\ph^{-1}\ps}.\de\ph,T_x\ell_{\ps^{-1}\ph}.\de x) \big\} 
\\&
= \inf\big\{G^1(x)_\ph(\de\ph,\de\ph) + G^2_x(\de x,\de x): (\de\ph,\de x)\in T_\ph\mathcal G\x T_x \mathcal M
\\&\qquad\qquad\qquad\qquad\qquad\qquad\qquad\qquad\qquad\qquad
\text{ with }\de \bar x = T_{(\ph,x)}\ell .(\de\ph,\de x) \big\} 
\\&
= \bar G^{(\ph,x)}_{\bar x}(\de\bar x,\de\bar x),
\end{align*}
which concludes the proof.
\end{proof}

\section{Optimal Metric Transport: a Toy Model}\label{sec:ORMT}
In this section we aim to generalize the $L^2$-Wasserstein metric from Optimal Mass Transport (OMT)  to the space of Riemannian metrics, which we will refer to as the Optimal Riemannian Metric Transport (ORMT) problem. In contrast to the standard OMT situation, the action of the diffeomorphism group on the space of Riemannian metrics is far from being transitive and thus we will restrict ourselves to Riemannian metrics that are elements of the same orbit under this group action. We will consider the general case in the next section, where we study an unbalanced version of this problem.

\subsection*{Optimal Riemannian Metric Transport on Compact Manifolds}
We start by considering the action of the diffeomorphism group $\Diff(M)$ on the space of Riemannian metrics $\Met(M)$ via pushforward, i.e., 
\begin{align}\label{pushforward}
\Met(M) \times \Diff(M)\to \Met(M), \qquad (g,\varphi)\mapsto \varphi_*g\;.
\end{align}
Note, that the infinitesimal version of this action is simply given by the negative of the Lie-derivative.
From here on we will assume that $M$ is equipped with a (base) Riemannian metric $g_0$ and we denote by $\operatorname{Orb}(g_0)$ the orbit of the diffeomorphism group through $g_0$, i.e.,
\begin{equation}
\operatorname{Orb}(g_0):=\left\{g=\varphi_*g_0\in \Met(M): \varphi\in \Diff(M)\right\}\;. 
\end{equation}
Furthermore, in the remainder of this section we will assume that 
$(M,g_0)$ admits no nontrivial isometries in $\Diff(M)$.

This setup allows us to introduce the following formulation of ORMT:
\begin{defn}[Optimal Riemannian Metric Transport]
Given a Riemannian metric $g_1\in \operatorname{Orb}(g_0)$ the ORMT problem consists of minimizing the Lagrangian 
\begin{align}\label{eq:generalMT}
\Lag(g) =  \int_0^1 \| v(t,\cdot)\|_{g_0, \vol(g)}^2 dt
\end{align} 
over all paths $g:[0,1]\to \Met(M)$ with $g(0)=g_0$ and $g(1)=g_1$
subject to the continuity equation 
\begin{equation}\label{eq:continuity_toy}
\dot g(t)= -\mathcal L_{v(t)} g(t)
\end{equation}
Here $\|\cdot\|_{g_0,\vol(g)}$ denotes some norm on the space of vector fields $\mathfrak X(M)$--depending on the source Riemannian metric $g_0$ and on the  volume density $\vol(g(t))$ of the transported Riemannian metric $g(t)$. 

The \emph{Wasserstein-ORMT} problem corresponds to~\eqref{eq:generalMT} with 
\begin{equation}
\|v\|^2_{g_0,\vol(g(t))} = \int_M g_0(v,v) \vol(g(t))
\end{equation}
being the $L^2$-product with respect to $\vol(g(t))$.  
\end{defn}
\begin{rem}[The continuity equation]
We note that the continuity equation~\eqref{eq:continuity_toy} ensures that $g(t)\in \Orb$ for all time $t$.
This can be easily seen as a  straightforward generalization of standard OMT, where the continuity equation  reads as 
\begin{equation}
\dot \rho(t)=- \mathcal L_{v(t)} \rho(t) 
\end{equation}
where $\rho=\vol(g)$.
\end{rem}

\begin{rem}[Otto's Riemannian metric]\label{rem:ottoGW}
A classical result of Otto~\cite{OttoPic} shows that the dynamic formulation of OMT gives  rise to a (formal) Riemannian submersion between the group of diffeomorphisms with the $L^2$-metric
\begin{equation}
\GDiff_{\operatorname{\varphi}}(v\circ\varphi,v\circ\varphi)=\int_M g_0(v,v)\varphi_*\vol(g_0), \qquad v\in\mathfrak X(M)
\end{equation}
and the space of probability densities equipped with the so-called Otto-metric $\GProb$. 

In our setting the corresponding geometric picture is even simpler:  A generic orbit $\Orb$ is bijective to the diffeomorphism group and, in analogy with optimal transport, we denote the corresponding (formal) Riemannian metric on it by $\GOrb$:
 \begin{equation}
\GOrb_g(-\mathcal L_vg,-\mathcal L_vg)=\int_M g_0(v,v) \vol(g), \qquad g\in \operatorname{Orb}(g_0),\; v\in \mathfrak X(M),
\end{equation}
where we identified the tangent space to the orbit $\operatorname{Orb}(g_0)$ with all symmetric two-forms of the form $-\mathcal L_vg$.
\end{rem}

In the above definition we have introduced a formulation of dynamic optimal transport on the space of Riemannian metrics.  In the next result we show that this can be really seen as a generalization of OMT, as it is related to this classical theory via a (formal) Riemannian submersion. 
\begin{thm}[A Riemannian submersion onto Wasserstein space.]\label{thm:riemsub_W}
Assume that $M$ has a total volume of one w.r.t the base Riemannian metric $g_0$.   Then the map 
        \begin{align}\pi: \begin{cases}  
        (\Orb,\GOrb) &\to (\operatorname{Prob}(M),\GProb) \\
          g&\mapsto \operatorname{vol}(g)\end{cases}
        \end{align}
     is a Riemannian submersion, where  $\operatorname{Prob}(M)$ denotes the space of all smooth probability densities equipped with the Wasserstein Riemannian metric.
\end{thm}
\begin{proof}
First, we note that the action of the diffeomorphism group does not change the total volume of the manifold $M$ and thus for any $g\in \Orb$ we have $\operatorname{vol}(g)\in \on{Prob}(M)$. That $\pi$ is indeed a submersion, i.e., that any $\rho\in \on{Prob}(M)$ is in the image of $\pi$, follows from the transitivity of the action of $\Diff(M)$ on $\on{Prob}(M)$ using the  equivariance of $\pi$, i.e., that $\pi(\varphi_*g_0)=\varphi_*\pi(g_0)$ for any $\varphi\in \Diff(M)$.

That $\pi$ is indeed a Riemannian submersion follows from the fact that $\GOrb$ depends on its foot point $g$ only via the volume density  $\rho=\operatorname{vol}(g)$. This leads, in turn, to an induced Riemannian metric on the space of probability densities, which is exactly the Wasserstein metric. 
\end{proof}
Now, that we have seen the relation of Wasserstein-ORMT to the Wasserstein geometry on the space of probability densities we would like to obtain an analogue of the celebrated results of Benamou-Brenier~\cite{benamou2000computational}, who showed that the solution to Wasserstein-OMT is determined  by the condition that the flow of the optimal vector field $v$ always takes values in the so-called polar cone, i.e., that $\varphi(t)=\nabla f(t)$ where $f$ is a time-dependent convex function. 
Due to the trivial nature of the geometry of the $L^2$-metric~\eqref{eq:l2metric} on $\operatorname{Diff}(M)$ (geodesics on the infinite dimensional Lie group follow geodesics on the finite dimensional manifold $M$), this can be used to reduce the Wasserstein-OMT problem to solving a Monge-Ampere equation for the endpoint mapping $f(1)$. Somewhat, surprisingly it turns out, that the situation is even easier in our situation, which is the content of the following theorem:

\begin{thm}
Assume that $(M,g_0)$ is geodesically convex and admits no nontrivial
isometries in $\Diff(M)$. Assume furthermore that
\[
g(t)=\varphi(t)_*g_0
\]
is a smooth path in $\operatorname{Orb}(g_0)$ such that, for every
$x\in M$, the curve
\[
t \mapsto \varphi(t)(x)
\]
is a geodesic in $(M,g_0)$ parametrized with constant speed.

Then $g(t)$ is a geodesic of the Wasserstein--ORMT metric. Conversely, every geodesic of the Wasserstein--ORMT metric is locally of this form.
\end{thm}
\begin{proof}
By the assumption that $(M,g_0)$ admits no nontrivial isometries in
$\Diff(M)$, the map
\[
\Diff(M)\to \Orb,\qquad \varphi\mapsto \varphi_*g_0,
\]
is injective. Hence it identifies $\Orb$ with $\Diff(M)$.
A curve $t\mapsto \varphi(t)$ is a geodesic for the $L^2$-metric on $\Diff(M)$  if and only if for every $x\in M$, the curve $t\mapsto \varphi(t,x)$ is a
constant-speed geodesic in $(M,g_0)$~\cite{ebin1970groups}. Since the map $\varphi\mapsto\varphi_*g_0$ is an isometry onto
$\Orb$, the path
\[
g(t)=\varphi(t)_*g_0
\]
is a geodesic of the Wasserstein--ORMT metric. The same argument gives
the converse locally.
\end{proof}

\subsection*{Summary of the Riemannian Submersion Results}
The following diagram describes the geometric relation of ORMT and OMT as described in Remark~\ref{rem:ottoGW} and Theorem~\ref{thm:riemsub_W}: 
\begin{equation}
\xymatrix@R5mm{
\left(\Diff(M),\GDiff\right) \ar[dr]_{\varphi_*\vol(g_0)}\ar[rr]_{\varphi_*g_0} &\ar[l]& \left(\Orb,\GOrb\right) \ar[dl]^{\vol(g)}          \\
&  \left(\Prob,\GProb\right) &                                           
}
\end{equation}
Note, that the mapping from $\Diff(M)$ to $\Orb$ is a bijective, Riemannian isometry, whereas the two other arrows are Riemannian submersions.

\subsection*{Optimal Metric Transport on $\mathbb R^d$}
Here we will consider the situation of $M$ being the non-compact manifold $\mathbb R^d$ equipped with the standard Euclidean metric~$g_0$. 
This situation is analytically slightly more complicated as we have to specify certain decay conditions for the spaces of Riemannian metrics and diffeomorphisms, cf. Section~\ref{sec:manifolds_mappings}, but after defining the right spaces it will lead to explicit characterizations and formulas for the ORMT problem.  We first define the space $\Met^{\on{flat}}(\mathbb R^d)$ of
all flat Riemannian metrics on $\mathbb R^d$ which differ from the Euclidean metric $g_0$ in an $H^\infty$-way.  We also recall that $\Diff(\mathbb R^d)$ was similarly defined as the space of all diffeomorphisms which are of the form $\on{Id}_{\mathbb R^d} + f$ where $f\in H^\infty(\mathbb R^d,\mathbb R^d)$.
  The advantage of this situation is that it allows us to precisely characterize the orbit $\Orb$.
\begin{thm}\label{thm:orbit}
Let $d\ge 2$. The orbit $\Orb$ of the diffeomorphism group under the action~\eqref{pushforward} through the Euclidean metric $g_0$ is the space $\Met^{\on{flat}}(\mathbb R^d)$  of all flat metrics on $\mathbb R^d$. Furthermore, the orbit $\Orb$
is again bijective to $\Diff(\mathbb R^d)$, i.e., for any $g\in \Met^{\on{flat}}$  there exists an unique $\varphi\in\Diff(\mathbb R^d)$ such that $\varphi_*g_0=g$.
\end{thm}
\begin{rem}
The situation is different but simpler when $d=1$. Then one can identify the manifold of metrics with the set of 
positive functions, and the image of the pullback is a co-dimension two subspace of this set, cf. \cite{bauer2014homogenous}.  
\end{rem}
The local version of this theorem is (for surfaces) well known since Gau\ss. The usual proof as given in \cite[24.7]{Michor08} can be adapted to the global version in the situation here and is equivalent to the proof below. 

\begin{proof}
We will show this result for the pull-back action instead of the push-forward action; the result for the push-forward action follows by replacing the obtained solution $\varphi$ with $\varphi^{-1}$.

For the  curvature we have $R^{\ph^*g_0} = \ph^* R^{g_0} =\ph^*0=0$, so the orbit 
consists of flat metrics. It remains to prove that each flat metric is in the orbit and that the corresponding diffeomorphism is unique. The uniqueness is clear, as there are no isometries of the Euclidean metric in $\Diff(\mathbb R^d)$; recall, that we require the elements of $\Diff(\mathbb R^d)$ to decay towards the identity.   

Next we prove the existence. Therefore let 
$g\in \Met^{\on{flat}}(\mathbb R^d)$ be a flat metric. 
Considering $g$ as a symmetric positive matrix, let
$s:= \sqrt{g}$. We search for an orthogonal matrix valued function $u\in C^\infty(\mathbb R^d,SO(d))$ 
such that $u.s = d\ph$ for a diffeomorphism $\ph$. 
For the following see \cite[Section 25]{Michor08}. 

Let $\si_i := \sum_j s_{ij}dx^j$ be the rows of $s$. Then for the metric we have 
$g = \sum_i \si_i\otimes \si_i$, thus the column vector $\si=(\si_1,\dots,\si_d)^t$ of 1-forms is a 
global orthonormal coframe. 
We want $u.\si=d\ph$, so the 2-form $d(u.\si)$ should vanish.
But 
$$
0=d(u.\si) = du\wedge \si + u.d\si \iff 0=u^{-1}.du\wedge \si + d\si
$$
This means that the $\mathfrak o(d)$-valued 1-form $\om:= u^{-1}.du$ is the connection 1-form for the 
Levi-Civita connection of the metric $g$. Since $g$ is flat, the curvature 2-form 
$\Om=d\om+\om\wedge \om$ vanishes. We consider now the trivial principal bundle 
$\on{pr}_1:\mathbb R^d\x SO(d)\to \mathbb R^d$ and the principal connection form $\on{pr}_1^*\om$ 
on it which is flat, so the horizontal distribution is integrable. Let 
$L(u_0)\subset \mathbb R^d\x SO(d)$ be the horizontal leaf through the point 
$(0,u_0)\in \mathbb R^d\x SO(d)$, then the restriction $\on{pr}_1:L(u_0)\to \mathbb R^d$ is a 
covering map and thus a diffeomorphism whose inverse furnishes us the required 
$u\in C^\infty(\mathbb R^d,SO(d))$ which is unique up to right multiplication by $u_0\in SO(d)$. 
The function $u:\mathbb R^d\to SO(d)$ is also called the Cartan development for $\om$. 

Thus $u.\si = d\ph$ for for a column vector $\ph = (\ph^1,\dots,\ph^d)$ of functions which 
defines a smooth map $\ph:\mathbb R^d\to \mathbb R^d$. Since $d\ph = u.\si$ is everywhere 
invertible, $\ph$ is locally a diffeomorphism. Since $g$ falls to $g_0=\mathbb I_n$ as a function in 
$\mathcal A$, the same is true for $\si$ and thus also for $u.\si$ since $u$ is bounded.
So $d(\ph - \on{Id}_{\mathbb R^d})$ is asymptotically 0,  thus $\ph - \on{Id}_{\mathbb R^d}$ is 
asymptotically a constant matrix $A$; here we need $n\ge 2$. Replacing $\ph$ by $\ph-A$ we see that $\ph$ then falls 
asymptotically towards $\on{Id}_{\mathbb R^d}$. Thus $\ph$ is a proper mapping and thus has closed 
image, which is also open since $\ph$ is still a local diffeomorphism. Thus 
$\ph\in \Diff(\mathbb R^d)$.

Finally, $d\ph^t.d\ph = (u.\si)^t.u.\si = \si^t.u^t.u.\si = \si^t.\si = g$. 
Note that $\ph$ is unique in $\Diff(\mathbb R^d)$. This is also clear from the fact,
that the fiber of the pullback action over $\ph^*g_0$ consists of all isometries of  $\ph^*g_0$ which 
is the group $\ph^{-1}\o (\mathbb R^d\ltimes O(d))\o \ph\subset \Diff(\mathbb R^d)$ whose intersection with 
$\Diff(\mathbb R^d)$ is trivial. 
\end{proof}

Using this the ORMT problem can be formulated as a transport problem on the space of flat Riemannian metrics: 
\begin{defn}[Optimal Riemannian Metric Transport on the Space of Flat Metrics]
Given a flat Riemannian metrics $g_1\in \Met^{\on{flat}}(\mathbb R^d)$ the ORMT problem consists of minimizing the Lagrangian 
\eqref{eq:generalMT}
over all paths $g:[0,1]\to \Met^{\on{flat}}(\mathbb R^d)$ with $g(0)=g_0$ and $g(1)=g_1$
subject to the continuity equation \eqref{eq:continuity_toy}. 
\end{defn}
\begin{rem}
Note that the ORMT problem  on the space of flat metrics is well defined as the action of the diffeomorphism group on the space of flat Riemannian metrics is transitive by Theorem~\ref{thm:orbit}. 
\end{rem}
In this setting the solution of the ORMT problem is even simpler:

\begin{cor}
Let $g_1\in \operatorname{Met}_{\operatorname{flat}}(\mathbb R^d)$ and let
$\varphi(1)=\operatorname{Id}+f$ be the unique diffeomorphism such that
\(
\varphi(1)_*g_0=g_1 .
\)
Assume that
\(
\varphi(t)=\operatorname{Id}+t f
\)
belongs to $\Diff(\mathbb R^d)$ for all $t\in[0,1]$. Then the path
\[
g(t)=\varphi(t)_*g_0
\]
is a geodesic of the Wasserstein--ORMT metric.
\end{cor}
Note, that the additional assumption that $\varphi(t)\in \Diff(\mathbb R^d)$ is automatically satisfied if $f$ is sufficiently
small in a $C^1$-norm.

\section{Unbalanced Metric Transport and the Wasserstein--Ebin Metric}\label{sec:OURMT}
In the previous section we studied a metric transport problem, where we had to restrict ourselves to Riemannian metrics that belong to the same orbit of the diffeomorphism group. This provides a severe limitation for any applications and thus we will now introduce the main model of the present article, in which we will overcome this by using ideas from the field of \emph{unbalanced} optimal mass transport.

\subsection*{A Dynamic Formulation}
We start by introducing the  dynamic formulation of Unbalanced Optimal Riemannian Metric Transport (UORMT):
\begin{defn}[Dynamic Unbalanced Optimal Riemannian Metric Transport]
Given Riemannian metrics $g_0,g_1\in \Met(M)$ the UORMT problem consists of minimizing the Lagrangian 
\begin{align}\label{eq:generalUMT}
\Lag(g) =  \int_0^1 \| v\|_{g_0, \vol(g)}^2 dt+
\int_0^1 G^{\Met}_g(h,h) dt
\end{align}
over all paths $g:[0,1]\to \Met(M)$ with $g(0)=g_0$ and $g(1)=g_1$
subject to the constraints 
\begin{equation}\label{eq:cont}
\dot g(t)= -\mathcal L_{v(t)} g(t) + h(t).
\end{equation}
Here $G^{\Met}$ is a given Riemannian metric on the space of all Riemannian metrics $\Met(M)$.
The norm $\|\cdot\|_{g_0,\vol(g)}$ on the space of vector fields $\mathfrak X(M)$ depends again on the source Riemannian metric $g_0$ and on the variable volume density $\vol(g)$.    
\end{defn}

\begin{rem}
In analogy with unbalanced  optimal mass transport we view $h$ as a source term that allows us to change the Riemannian metric without transporting it, i.e., without applying the action of the diffeomorphism group. Furthermore we note that the continuity equation~\eqref{eq:cont} is a direct generalization of the 
continuity equation of unbalanced OMT. In fact letting $\rho(t)=\vol(g(t))$ one obtains from~\eqref{eq:cont} the continuity equation
\begin{equation}
\dot \rho(t)=- \mathcal L_{v(t)}\rh(t) + \frac 12 f(t)\rho \,,
\end{equation}
where $\rho=\vol(g)$ and $f=\operatorname{tr}(g^{-1}h)$. Note, that this differs only by the factor $\frac12$ from the standard continuity equation of UOMT.
\end{rem}

Next we want to specify the particular choice of norm on $\mathfrak X(M)$ and Riemannian metric on $\Met(M)$, where we aim to generalize the Wasserstein--Fisher--Rao setting of UOMT to our situation:
\begin{defn}[Wasserstein--Ebin Unbalanced Optimal Riemannian Metric Transport]
Let $\const\geq 0$. The Wasserstein--Ebin UORMT Lagrangian corresponds to~\eqref{eq:generalUMT} with the norm $\|\cdot\|_{g_0,\rho(t)}$ being the $L^2$-product~\eqref{eq:l2metric} and $G$ being (a multiple of) the invariant $L^2$ (Ebin) metric $G^{\on{E}}$ on $\Met(M)$, i.e.,  
\begin{equation}\label{EqWE}
\Lag^{\operatorname{WE}}(g)=\int_0^1\int_M  g_0(v,v) \vol(g) dt+
\frac{d\const}{4} \int_0^1\int_M \operatorname{tr}(g^{-1}hg^{-1}h)\operatorname{vol}(g) dt \,,
\end{equation}
subject to the continuity equation~\eqref{eq:cont}. 

We refer to the corresponding (formal) Riemannian metric $G^{\operatorname{WE}}$ on $\Met(M)$ as the Wasserstein--Ebin metric, i.e.,
\begin{align}\label{eq:Wasserstein--Ebin-metric}
G^{\operatorname{WE}}_g(\delta g, \delta g)=\inf_{v,h}\int_M  g_0(v,v) \vol(g)+
\frac{d\const}4 \int_M \operatorname{tr}(g^{-1}hg^{-1}h)\operatorname{vol}(g) 
\end{align}
subject to the static continuity equation
\begin{equation}
\delta g= -\mathcal L_{v} g + h, \qquad v\in\X(M),\; h\in T_g\Met(M).
\end{equation}
\end{defn}
\begin{rem}[Analogy to the Wasserstein--Fisher--Rao metric]
Note, that in analogy with the Wasserstein--Fisher--Rao unbalanced optimal mass transport framework, the first term corresponds to a non-invariant metric on the diffeomorphism group, whereas the second term is an invariant Riemannian metric on the source term, i.e., in the spirit of Proposition~\ref{prop:infimalconv} it can be interpreted as the infimal convolution of the Wasserstein metric on $\Diff(M)$ and the Ebin metric on $\Met(M)$.
\end{rem}

\begin{rem}[Alternative formula for the Wasserstein--Ebin metric]
    Using the continuity equation we can express $h=\delta g +\mathcal L_{v} g$ and thus we have that
\begin{align}
G^{\operatorname{WE}}_g(\delta g, \delta g)&=\inf_{v,h}\int_M  \left(g_0(v,v) +
\frac{d\const}4  \operatorname{tr}(g^{-1}hg^{-1}h)\right)\operatorname{vol}(g)\\
&=\inf_{v}\int_M \left( g_0(v,v) +
\frac{d\const}4 \operatorname{tr}(g^{-1}(\delta g +\mathcal L_{v} g)g^{-1}(\delta g +\mathcal L_{v} g))\right)\operatorname{vol}(g)\\
&:=\inf_{v} F(g,\delta g,v),
\end{align}
i.e., we have written the Wasserstein--Ebin metric as single infimum for the (transport) vector field $v$.  Finally we note, that 
\begin{align*}
F(g,\delta g,v) &= \int_M g_0(v,v) \vol(g) +\frac{d\const}4 \int_M\operatorname{tr}(g^{-1}.\mathcal L_{v} g. g^{-1}.\mathcal L_{v} g)\operatorname{vol}(g)\\&\qquad +\frac{d\const}2 \int_M\operatorname{tr}(g^{-1}.\mathcal L_{v} g.g^{-1}. \delta g)\operatorname{vol}(g)+ \frac{d\const}4 \int_M\operatorname{tr}(g^{-1}\delta g g^{-1}\delta g)\operatorname{vol}(g)
\end{align*}
is a quadratic functional in $v\in \X(M)$ whose second order term is a symmetric weak positive definite bilinear form on $\X(M)\x \X(M)$. If there exists a minimum $v^0 = v(g,\delta g)$, it is thus unique and will be attained where the derivative with respect to $v$ vanishes.
\end{rem}

\subsection*{Relation to the Wasserstein--Fisher--Rao Metric}
 Next we will show that the Wasserstein--Ebin metric is indeed connected (via a Riemannian submersion) to the Wasserstein--Fisher--Rao metric of unbalanced OMT:
\begin{thm}[A Riemannian submersion onto the Wasserstein--Fisher--Rao metric.]\label{thm:riemsub_WFR}
    The map 
        \begin{align}\label{eq:pi1}
        \pi_1:\begin{cases}(\operatorname{Met}(M),G^{\operatorname{WE}}) &\to (\operatorname{Dens}(M),G^{\operatorname{WFR}}) \\
             g&\mapsto \operatorname{vol}(g)
             \end{cases}
        \end{align}
     is a Riemannian submersion, where  $\operatorname{Dens}(M)$ denotes the space of all smooth densities. 
\end{thm}
\begin{rem}\label{rem:naturalityvol}
Observe that the mapping $\pi_1:\Met(M)\to \on{Dens}(M)$ is a natural transformation for the action of the diffeomorphism group, i.e., $\pi_1(\ph^*g)=\vol(\ph^*g) = \ph^*\vol(g)=\ph^*\pi_1(g)$. Note, that this property is merely reflecting the statement that $\vol$ transforms correctly under chart changes. The infinitesimal version of this is that $\pi_1$ commutes according to the chain rule with Lie derivatives:
\begin{equation}\label{eq:naturality_vol}
\begin{aligned}
\L_v\vol(g) &= \p_t|_0 (\on{Fl}^v_t)^*\vol(g) 
=\p_t|_0 \vol((\on{Fl}^v_t)^*g)
\\&
=d\vol(g)(\p_t|_0 ((\on{Fl}^v_t)^*g) 
= \tfrac12\on{tr}(g^{-1}.\L_vg)\vol(g)\,,
\end{aligned}
\end{equation}
where we used the well known formula 
$T_g\vol.\delta g = \tfrac12 \on{Tr}(g^{-1}.\delta g)\vol(g)$; see e.g.~\cite[12.9]{Michor80} for a proof.    
\end{rem}
\begin{proof}[Proof of Theorem~\ref{thm:riemsub_WFR}]
To prove that $\pi_1$ is a Riemannian submersion we aim to show that for $\rho\in \on{Dens}(M)$,  $g\in\pi_1^{-1}(\rho)$ and 
$\delta \rho\in T_\rho\on{Dens}(M)$ we have
\begin{equation}\label{eq:riemsub_WE-WFR}
G^{\operatorname{WFR}}_{\rho}(\delta \rho,\delta \rho)=
\inf \left\{G^{\on{WE}}_g(\delta g,\delta g) \,; \delta g \in T_g \Met(M) \text{ with } T_g\pi_1.\delta g = \delta \rho\right\}\,.
\end{equation}
We start by fixing a Riemannian metric $g\in \pi_1^{-1}(\rho)$ and will prove the independence on the choice of $g$ afterwards. 

Next, we use that any tangent vector $h\in T_g\Met(M)$ 
can be decomposed as $h=z+\tfrac{r}{d}g$, where $r$ is a function on $M$, where $d$ is the dimension of $M$ and where $z$ is trace free
w.r.t. $g$, i.e., $\on{tr}(g^{-1}z)=0$. Note, that $r$ can be explicitly chosen as $r=\operatorname{tr}(g^{-1}h)$. 
Using this decomposition for the source term $h$ the continuity equation reads as
\begin{equation}
    \delta g= -\mathcal L_{v} g + z +  \tfrac{r}{d} g.
\end{equation}
Now we aim to solve the infimum in equation~\eqref{eq:riemsub_WE-WFR}; i.e., we aim  to minimize over all $\delta g$ such that 
\begin{equation}
T_g \pi_1(\delta g)=\delta \rho=-\L_v\rho+f\rho
\end{equation}
Using $\vol(g)=\rho$,  formula~\eqref{eq:naturality_vol} and the variation formula for $\vol$ we obtain
\begin{align}
T_g \pi_1(-\L_v g+z + \tfrac{r}{d}   g)&=-\L_v\rho+\frac12\on{tr}(g^{-1}z)\rho +\frac 1{2d} \on{tr}(g^{-1}rg)\rho\\&=
-\L_v\rho+0+\frac{r}2 \rho\;,
\end{align}
which implies that $r$ is uniquely determined via $r=2f$.
Thus we can express the infimum~\eqref{eq:riemsub_WE-WFR} via
\begin{equation}~\label{eq:riemsub_WE-WFR2}
\inf_{\delta g \in T_g \Met(M)} \left\{G^{\on{WE}}_g(\delta g,\delta g) \,: T_g\pi_1.\delta g = \delta \rho\right\}
=\inf_{z,v} \left\{G^{\on{WE}}_g(\delta g,\delta g) \,: \delta g=-\L_v g+z+\tfrac{2f}{d} g\right\}
\end{equation}
Since the decomposition $h=z+\tfrac{2f}{d} g$ is orthogonal with respect to the Ebin metric 
and since \(z\) is free in the infimum in~\eqref{eq:riemsub_WE-WFR2}, the minimum is attained for \(z=0\). Thus we have
\begin{align}
&\inf_{v,f,z}\int\left( \|v\|^2_{g_0}\rho+  \tfrac{d\const}{4}\on{tr}(g^{-1}zg^{-1}z)+ \tfrac{d\const}{4d^2}\on{tr}(g^{-1}g2fg^{-1}g2f)\right)\vol(g)\\&\qquad=
\inf_{v,f} \int\left( \|v\|^2_{g_0}\rho+0+ \tfrac{\const}{d}\on{tr}(g^{-1}gfg^{-1}gf)\right)\vol(g)\\
  &\qquad= \inf_{v,f} \int\left( \|v\|^2_{g_0}\rho+ \const f^2\right)\rho
  \\&\qquad = G^{\on{WFR}}_{\rho}(\delta \rho,\delta \rho).
\end{align}
It only remains to show the independence on the  choice of $g\in\pi_1^{-1}(\rho)$, which is, however, clear as the calculation above depends only via $\rho=\vol(g)$ on~$g$.
\end{proof}

\subsection*{Otto's Riemannian Submersions for Unbalanced Metric Transport}
In this section we want to develop the equivalent of Otto's Riemannian submersion for OMT in our situation, i.e., for the UORMT problem. 
We start by introducing a certain extension of the diffeomorphism group, which will play a central role in the desired geometric picture. 
\begin{defn}[The automorphism and gauge groups on $TM$]\label{semidirect} 
We consider the group $\on{Aut}(TM)$ of all vector bundle automorphisms of the tangent bundle 
$\pi_M:TM\to M$; i.e., 
\begin{align}
\on{Aut}(TM) =\big\{\bar\ph:TM\to TM &\text{ fiberwise linear and invertible such that: } 
\\&
\pi_M\o \bar \ph = \ph\o \pi_M\text{ for }\ph\in\Diff(M)\big\}
\end{align}
We also consider the normal subgroup 
\begin{align}
\on{Gau}(TM) = \big\{\al\in\on{Aut}(TM): \pi_M\circ\al=\on{Id}_M \big\} = \Ga(M \leftarrow GL(TM))
\end{align}
so that we get a semidirect product splitting exact sequence 
$$\xymatrix{
\{\on{Id}_{TM}\} \ar[r] & \on{Gau}(TM) \ar[r]^{i} & \on{Aut}(TM) \ar@/^/[r]^{p:\bar\ph\mapsto \ph} & \Diff(M) \ar@/^/[l]^{T\ph\gets \ph:j} \ar[r] &\{\on{Id}_M\}\,.
}$$
The semidirect product structure on $\on{Gau}(TM)\rtimes \Diff(M)$ and the action on $TM$ are then given by  the isomorphism $(\al,\ph)\mapsto \al\o T\ph\in\on{Aut}(TM)$ which leads to 
\begin{align}
(\al,\ph)(\al',\ph') &= (\al.(T\ph.\al'.T\ph^{-1}), \ph\o \ph'),
\\
(\al,\ph)^{-1} &= (T\ph^{-1}.\al^{-1}.T\ph, \ph^{-1}) \label{Eqsemidirectleft}\\
(\al,\ph)(X_x) &= \al_{\ph(x)}.T_x\ph.X_x.
\end{align}
\end{defn}
Less natural but better suited for our purposes below is the other way  to express the semidirect product structure as $\Diff(M)\ltimes \on{Gau}(TM)$ is induced by the isomorphism $(\ph,\al)\mapsto T\ph\o\al\in\on{Aut}(TM)$ which leads to 
\begin{align}
(\ph,\al)(\ph',\al') &= (\ph\o \ph', (T\ph')^{-1}.\al.T\ph'.\al'),
\\
(\ph,\al)^{-1} &= (\ph^{-1},T\ph.\al^{-1}.T\ph^{-1}) \label{Eqsemidirectright}\\
(\ph,\al)(X_x) &= T_x\ph.\al_{x}.X_x.
\end{align}
Finally the action of $\on{Aut}(TM)$ on the space of metrics is defined by
\begin{equation}
    (\ph,\al)_*g \coloneqq \ph_* \al_* g = \ph_*(\al^{-T}g \al^{-1})\,,
\end{equation}
which induces a corresponding action on the space of densities via
\begin{equation}
    (\ph,\al)_* \rho \coloneqq \ph_*(|\operatorname{det}(\alpha)|^{-1} \rho)\,.
\end{equation}
We are now ready to state the main result of this section, where 
from now on we use the left action to define the semi direct product:
\begin{thm}[Otto's Riemannian submersion for the Wasserstein--Ebin metric.]\label{thm:Otto_WE}
    The map 
        \begin{align}\label{eq:pi2}
\pi_2:
\begin{cases}
  \on{Aut}(TM)\to\Met(M)\\  (\ph,\alpha)\mapsto \ph_*\alpha_*g_0=\ph_*(\alpha^{-T}g_0\alpha^{-1})
\end{cases}
        \end{align}
     is a Riemannian submersion, where  $\Met(M)$ is equipped with the Wasserstein--Ebin metric~\eqref{eq:Wasserstein--Ebin-metric} and where $\on{Aut}(TM)$ is equipped with the metric
\begin{equation}\label{eq:met_auttm}
\begin{aligned}
    &G^{\on{Aut}(TM)}_{(\ph,\al)}((\delta \ph, \delta \alpha),(\delta \ph, \delta \alpha)) 
    \\&\qquad
    = \int \left(g_0(\delta \ph,\delta \ph) +d\Lambda\on{Tr}(g_0 \alpha^{-1} \delta \alpha g_0^{-1} \delta \alpha^\top \alpha^{-\top}) \right)|\on{det}(\alpha)|^{-1} \vol(g_0)\;.
\end{aligned}
\end{equation}
\end{thm}
Note, that the second summand in the integral in the definition of the metric $G^{\on{Aut}(TM)}$ comes from the left invariant Riemannian metric on $GL(T_xM)$ given by 
$$
G^{GL(T_xM),g_0}_A(X,X) = \on{tr}(A^{-1} X  g_0^{-1}(A^{-1} X)^Tg_0) = \on{tr}(A^{-1} X  g_0^{-1} X^TA^{-T}g_0),
$$ 
where $A^{-T} = (A^{-1})^T = (A^T)^{-1}:T_x^*M\to T_x^*M$.
To prove the above result we shall need the following lemma, which we will use to calculate the variational derivative of the projection $\pi_2$.
\begin{lem}\label{lem:Lie}  
Let $\ph(t)$ be a smooth curve of diffeomorphisms on a manifold $M$ locally
defined for each $t$, with $f_0=\on{Id}_M$. We consider the two time-dependent vector fields 
\begin{equation}
X(t)(x) := (T_x\ph(t))^{-1} \p_t \ph(t)(x),\qquad
Y(t)(x) := (\p_t \ph(t))(\ph(t)^{-1}(x)).
\end{equation}
Then $T(\ph(t)).X(t)=\p_t\ph(t) = Y(t)\o \ph(t)$,
so $X(t)$ and $Y(t)$ are $\ph(t)$-related. 
Moreover, for any  tensor field $K$  
on $M$ we have:
\begin{align}
\label{eq:Lie}
 \p_t\ph(t)^*K &= 
       \ph(t)^*\mathcal{L}_{Y(t)}K = \mathcal{L}_{X(t)}\,\ph(t)^*K\,. 
\\ \label{eq:Lieinverse}
\p_t\ph(t)_*K &= \p_t(\ph(t)^{-1})^*K = 
 -\ph(t)_*\mathcal{L}_{X(t)}K = -\mathcal{L}_{Y(t)}\,\ph(t)_*K\,.
\end{align} 
\end{lem}

\begin{proof} 
A more involved proof for general natural bundles can be found in \cite{Michor25}. Here we shall present a simpler proof, which is targeted to the setting of this article. Note, that the result is well known when $K$ is a differential form; a proof can be found in \cite[31.11]{Michor08} and for natural bundles (but not for non-autonomous vector fields) in \cite[8.15 -- 8.20]{Michor08}. 
It is also obviously true if $K$ is a vector field. 
Since both operators 
$$\ph(t)^*\mathcal{L}_{Y(t)}\text{ and } \mathcal{L}_{X(t)}\,\ph(t)^*: \Ga(T^\ell_k M)\to \Ga(T^\ell_k M)$$  
are derivations over the tensor algebra homomorphism $\ph(t)^*$ and since the tensor algebra is generated by $C^{\infty}(M,\mathbb R)$, $\Om^1(M)$, and $\X(M)$, these operators agree on the whole tensor algebra and \eqref{eq:Lie} follows.

To show \eqref{eq:Lieinverse} note first that 
\begin{align}
  &0 =  \p_t(\on{Id}) = \p_t(\ph(t)^{-1}\o \ph(t)) = (\p_t\ph(t)^{-1})\o \ph(t) + T\ph(t)^{-1}\o \p_t\ph(t)
\\&
\p_t(\ph(t)^{-1}) = - T\ph(t)^{-1}\o (\p_t\ph(t)) \o \ph(t)
\\&
T\ph(t)\o \p_t(\ph(t)^{-1}) = -(\p_t\ph(t))\o \ph(t)^{-1} = - Y(t)
\\&
(\p_t\ph(t)^{-1})\o \ph(t) = - T\ph(t)^{-1}\o (\p_t\ph(t)) = -X(t)
\end{align}
Hence, replacing $\ph(t)$ by $\ph(t)^{-1}$ in \eqref{eq:Lie} replaces $X(t)$ by $-Y(t)$ and $Y(t)$ by $-X(t)$ and transforms \eqref{eq:Lie} into \eqref{eq:Lieinverse}.
\end{proof}

\begin{proof}[Proof of Theorem~\ref{thm:Otto_WE}]
 To show the surjectivity of $\pi_2$ we note that $\on{Gau}(TM)$ alone acts transitively from the left on the space $\Met(M)$ of all Riemannian metrics on $M$ by fiberwise algebraic push forward, $g\mapsto g(\al^{-1}\cdot,\al^{-1}\cdot)$, or by $g\mapsto \al^{-T}\o g\o\al^{-1}\in L(TM,T^*M)$, where $\al^{-T}$ is shorthand for $(\al^{-1})^T=(\al^T)^{-1}$,
also the full vector bundle automorphism group $\on{Aut}(TM)$ acts transitively. Thus the mapping $\pi_2$ is also surjective. 

To prove that $\pi_2$ is a Riemannian submersion we aim to show that for $g\in \Met(M)$,  $(\ph,\alpha)\in\pi_2^{-1}(g)$, i.e., $g=\ph_*(\al^{-T}.g_0.\al^{-1})$, and 
$\delta g \in T_{g}\Met(M)$ we have
\begin{multline}\label{eq:riemsub_Aut-Met}
G^{\on{WE}}_{g}(\delta g,\delta g)=
\inf\Big\{G^{\operatorname{Aut}(TM)}_{(\ph.\al)}((\de \ph,\de\al),(\de \ph,\de\al)) \,; 
\\
(\de \ph,\de\al)\in T_{(\ph,\al)}\on{Aut}(TM),T_{(\ph,\al)}\pi_3.(\de \ph,\de\al) = \delta g\Big\}\,.
\end{multline}
We start by fixing a point $(\ph,\alpha) \in \pi_2^{-1}(g)$ and will prove the independence on the choice of the point in the preimage afterwards. 

Next we calculate the tangent mapping of $\pi_2$. We shall need that $D_{\al,\de\al} \o \ph_* = \ph_*\o D_{\al,\de\al}$. To see this let $\al(s)$ be a smooth curve in $\on{Gau}(TM)$.
Then for $\de\al = \p_s\al(s)$ and a smooth fiber respecting  expression $F(\al(s))$ we have 
$\p_s\ph_*F(\al(s)) = \ph_*\p_s F(\al(s))$ since $\ph_*$ is bounded (=smooth)
linear. Thus we may use Lemma \ref{lem:Lie} as follows: let $v=\de\ph \o \ph^{-1} \in \X(M).$
Then
 \begin{align}
&T_{(\ph,\al)}\pi_2(\de\ph,\de\al)=
D_{(\ph,\al),(\de\ph,\de\al)}(\ph,\al)_*g_0 \\&\qquad= 
(D_{\ph,\de\ph} + D_{\al,\de\al})\big(\ph_*(\al^{-T}.g_0. \al^{-1})\big) =
\\&\qquad
= -\L_vg + \ph_*\big(D_{\al,\de\al}(\al^{-T}.g_0. \al^{-1})\big)
\\&\qquad
= -\L_vg+\ph_*\Big(-\al^{-T}.(\de\al)^T.\al^{-T}.g_0. \al^{-1} - \al^{-T}.g_0. \al^{-1}.\de\al.\al^{-1}
\Big)\\
\label{eq:tangentpi2}
&\qquad=
-\L_{v}g-\ph_*\left(\al^{-T}.(\de\al)^T\right).g + g.\ph_*\left(\de\al.\al^{-1}\right)
\end{align}
We can rewrite the infimum~\eqref{eq:riemsub_Aut-Met} as an infimum over all 
$v\in \mathfrak X(M)$ and $\de\al\in T_{\alpha}\on{Gau}(TM)$ such that 
\begin{equation}\label{eq:contequation_pi3}
-\L_{v}g-\ph_*\left((g.\de\al.\al^{-1})^t + g\de\al.\al^{-1}\right)=\delta g=: -\L_{v}g- \ph_*h
\end{equation}
where $h = (g.\de\al.\al^{-1})^t + g\de\al.\al^{-1}$
is a symmetric bilinear form. 
We consider the decomposition of the bilinear form $g.\de\al.\al^{-1} =: \tfrac12 h + k$ 
into its symmetric and skew symmetric parts and consider $\widetilde{\de\al}= g^{-1}.\tfrac12 h.\al$ and 
$\de\be := g^{-1}.k.\al \in T_{\alpha}\on{Aut}(TM)$.
We have pointwise
\begin{align}
\on{tr}(g_0.\al^{-1}.\widetilde{\de\al}.g_0^{-1}.\de\be^T.\al^{-T}) 
&= \on{tr}(\al^{-T}.g_0.\al^{-1}.\widetilde{\de\al}.\al^{-1}.\al.g_0^{-1}.\al^T.\al^{-T}.\de\be^T) 
\\
&= \on{tr}(g.\widetilde{\de\al}.\al^{-1}.g^{-1}.\al^{-T}\de\be^T)\label{eq:eq:contequation_pi3-2}
\\
&= \on{tr}(\tfrac12 h.g^{-1}.\al^{-T}(g^{-1}.k.\al)^T)
\\
&= \on{tr}(\tfrac12 h.g^{-1}.\al^{-T}\al^T.k^t.g^{-1})
\\
&= \tfrac12\on{tr}(g^{-1}.h.g^{-1}.k) =0
\end{align}
since $h^t=h$ and $k^t=-k$.
 Thus the infimum in~\eqref{eq:riemsub_Aut-Met} is attained for $\delta \beta=0$ and the optimal $\delta \alpha$ can be written as $\de\al=\widetilde{\de\al}=g^{-1}.\tfrac12 h.\al$.
It remains to calculate the $G^{\operatorname{Aut}(TM)}$ inner product for such $\delta \alpha$.  We get
\begin{align}
&G^{\operatorname{Aut}(TM)}_{\bar \ph}((0,\delta \alpha),(0,\delta \alpha))
= \tfrac{d\Lambda}{4}\int_M\operatorname{Tr}(g^{-1}hg^{-1} 
h)  \vol(g)
=\tfrac{d\Lambda}{4}G^{\on{E}}_g(h,h)
\end{align}
by using the computation in \eqref{eq:eq:contequation_pi3-2},
which implies that $\pi_2$ is a Riemannian submersion.
 \end{proof}


\subsection*{A Summary Theorem: A Commutative Diagram of Riemannian Submersions}
In the following main Theorem with give a summary of the all relations (Riemannian submersion) between classical unbalanced optimal transport and the proposed unbalanced metric transport model:
\begin{thm}[A commutative diagram of Riemannian submersions]\label{thm:diagram}
The following is a commutative diagram of Riemannian submersions:
\begin{equation}\label{bigdiagram}
\xymatrix@C=2cm@R=2cm{
(\on{Aut}(TM),G^{\on{Aut}(TM)}) 
\ar[r]^{\pi_2}_{(\ph,\alpha)\mapsto (\ph, \frac1{\sqrt{|\on{det}\alpha^{-1}|}})} 
\ar[d]_{\pi_4}^{(\ph,\alpha)\mapsto (\ph,\alpha_*g_0)\quad} 
\ar@/_5.5pc/[dd]_<<<<<<<<{\pi_3}|>>>>>>>{(\ph,\alpha)\mapsto \ph_*\alpha_*g_0} 
& (\on{Aut}(\mathcal{C}(M)),G^{\on{Aut}(\mathcal{C}(M))}) \ar[d]^{\pi_8}_{(\ph,\la)\mapsto (\ph,\la^2\rh_0)}
\ar@/^6pc/[dd]^<<<<<<<<<{\pi_0}|>>>>>>>{(\ph,\lambda)\mapsto \ph_*(\lambda^2\rho_0)}
\\
(\Diff(M)\!\x\! \Met(M),\! G^{W\!\x\! E}) \ar[d]_{\pi_5}^{(\ph,g)\mapsto \ph_*g} \ar[r]^{\pi_6}_{(\ph,g)\mapsto (\ph,\vol(g)) }
&(\Diff(M)\!\x\!\on{Dens}(M),\!G^{W\x FR}) \ar[d]^{\pi_7}_{(\ph.\rh)\mapsto \ph_*\rh}
\\
(\on{Met}(M),G^{\on{WE}}) \ar[r]_{\pi_1}^{g\mapsto\vol(g)}
&(\on{Dens}(M),G^{\on{WFR}})
}
\end{equation}
where the corresponding Riemannian metrics are given as follows:
\begin{align}
    &G^{\on{Aut}(TM)}_{(\ph,\al)}((\delta \ph, \delta \alpha),(\delta \ph, \delta \alpha)) =
    \\&
    = \int \left(g_0(\delta \ph,\delta \ph) +d\Lambda\on{Tr}(g_0 \alpha^{-1} \delta \alpha g_0^{-1} \delta \alpha^\top \alpha^{-\top}) \right)|\on{det}(\alpha)|^{-1} \vol(g_0)  
    \\&
 G^{W\x E}_{(\ph, g)}((\de\ph,h), (\de\ph,h)) 
= \int_M g_0(\de\ph,\de\ph)\vol(g)
+ \tfrac{d\const}{4} \int_M \on{tr}(g^{-1} h g^{-1} h)\vol(g)
\\&
G^{\on{WE}}_g(\delta g, \delta g)=\inf_{v,h}\int_M g_0(v,v) \vol(g)+
\tfrac{d\const}{4} \int_M \on{tr}(g^{-1}hg^{-1}h)\on{vol}(g) 
\\&\qquad\qquad
\text{subject to: }\quad \delta g= -\mathcal L_{v} g + h
\\& 
G^{\on{Aut}(\mathcal{C}(M))}_{\varphi,\lambda}((\delta \varphi,\delta \lambda),(\delta \varphi,\delta \lambda))= \int_M \big(\lambda^2 g_0(\delta \varphi,\delta \varphi) + \const(\delta \lambda)^2\big) \on{\vol}(g_0)
\\&
G^{W\x FR}_{\ph,\rh}((\de\ph,\de\rh)(\de\ph,\de\rh))
= \int_M g_0(\de\ph,\de\ph)\rh + \tfrac{d\const}{4}\int_M \Big(\frac{\de\rh}{\rh}\Big)^2\rh
\\&
G^{\on{WFR}}_\rho(\delta \rho,\delta \rho)=\inf_{v,f}\int_M  g_0(v,v) \rho +\const\int_M f^2\rho 
\\&\qquad\qquad
\text{subject to: }\quad \delta \rho= -\mathcal L_{v} \rho + f\rho . 
\end{align}
\end{thm}
For the proof of this Theorem we will make repeatedly use of the following Lemma concerning the composition of Riemannian submersions:
\begin{lem}\label{lem:RiemSub} Consider surjective submersions (so that $T_m\ps^i$ is a surjective linear map for each $m$) between possibly infinite dimensional weak Riemannian manifolds: 
\begin{equation}
\xymatrix{
(\mathcal M^1, G^1 ) \ar[r]^{\ps^1} & (\mathcal M^2, G^2) \ar[r]^{\ps^2} & (\mathcal M^3, G^3)
}
\end{equation}
If both $\ps^1$ and $\ps^2$ are Riemannian submersions, then also $\ps^2\o\ps^1$ is one.
\\
If $\ps_1$ and $\ps_2\o \ps_1$ are Riemannian submersions, then so is $\ps_1$.
\end{lem}

\begin{proof} The first conclusion is obvious, For the second one, the core of the argument is as follows: 
Let $V^1 \xrightarrow{\rh^1} $ and $ V^2 \xrightarrow{\rh^2} V^3$ be  surjective bounded linear maps between normed vector spaces such that $\ps_1$ and  $\rh^2\o \rh^1$ are normed quotient maps: That means means that for all $v^2\in V^2$ and $v^3\in V^3$ we have
\begin{align}
\|v^2\|_2 &= \inf\{ \|v^1\|_1: v^1\in V^1, \rh^1.v^1= v^2\}, \quad\text{ and }
\\
\|v^3\|_3 &= \inf\{ \|v^1\|_1: v^1\in V^1, \rh^2.\rh^1.v^1= v^3\}, \quad\text{ but then  }
\\
\inf\{ \|v^2\|_2&: v^2\in V^2, \rh^2.v^2= v^3\} = 
\\&
= \inf\{\inf\{ \|v^1\|_1: v^1\in V^1, \rh^1.v^1= v^2\}:v^2\in V^2, \rh^2.v^2=v^3\}
\\&
= \inf\{ \|v^1\|_1: v^1\in V^1, \rh^2.\rh^1.v^1= v^3\} = \|v^3\| \qedhere
\end{align}
\end{proof}

\begin{proof}[Proof of Theorem~\ref{thm:diagram}]
That $\pi_0$ is a Riemannian submersion has been shown in~\cite{GALLOUET20184199}, that  $\pi_1$ is a Riemannian submersion
is the content of Theorem~\ref{thm:riemsub_WFR}, and 
that $\pi_2$ is a Riemannian submersion was shown in Theorem~\ref{thm:Otto_WE}.

\paragraph{{\bf $\pi_3$ is a Riemannian submersion:}} Next we prove the result for $\pi_3$. It is clear that $\pi_3$ is surjective. To prove that $\pi_3$ is a Riemannian submersion it thus remains to show that for $\bar \ph=(\ph,\la)\in \operatorname{Aut}(\mathcal{C}(M))$,  $\tilde \ph=(\ph,\alpha)\in\pi_3^{-1}(\bar \ph)$ and 
$\delta \tilde \ph=(\delta \ph,\delta\lambda) \in T_{\tilde \ph}\operatorname{Aut}(\mathcal{C}(M))$ we have
\begin{equation}\label{eq:riemsub_Aut-Aut}
G^{\operatorname{Aut}(\mathcal{C}(M)))}_{\tilde \ph}(\delta\tilde \ph,\delta\tilde \ph)=
\inf_{\delta\bar \ph\in T_{\bar \ph}\operatorname{Aut}(TM)} \left\{G^{\operatorname{Aut}(TM)}_{\bar \ph}(\delta \bar \ph,\delta\bar\ph) \,; T_{\bar \ph}\pi_3.(\delta \bar \ph) = \delta \tilde\ph\right\}.
\end{equation}
Therefore we start by fixing a point $\bar \ph= (\ph,\alpha) \in \pi_3^{-1}(\ph,\lambda)$. We will prove the independence on the choice of the foot point in the pre-image afterwards. Note, that it immediately follows that the first component of $\bar \ph$ and $\tilde \ph$ are the same as the first component of $\pi_3$ is merely the identity.

Next, we calculate the tangent mapping of $\pi_3$:
\begin{equation}\label{eq:tangentpi3}
T_{(\ph,\alpha)}\pi_3.(\delta \ph,\delta \alpha)=
(\delta \ph,-\frac{1}{2}|\operatorname{det}(\alpha)|^{-1/2} \on{tr}(\alpha^{-1}\delta\alpha)),
\end{equation}
where we used Jacobi's formula for the derivative of the determinant. Thus the infimum in~\eqref{eq:riemsub_Aut-Aut} has to be calculated over $\delta \alpha\in T_{\alpha}\on{Gau}(TM)$ only, as the $\delta \phi$ is uniquely determined by~\eqref{eq:tangentpi3}. Next we decompose  $\delta \alpha$ into its trace free and pure diagonal parts via $\delta \alpha=\delta \beta+\frac1d f\alpha$, where $d=\operatorname{dim}(M)$ and $\operatorname{Tr}(\alpha^{-1}\delta\beta)=0$. Note that $f$ can be explicitly calculated via $f=\on{tr}(\alpha^{-1}\delta\alpha)$. Thus we have 
\begin{equation}
T_{(\ph,\alpha)}\pi_3.(\delta \ph,\delta \beta+\frac1d f\alpha)=
(\delta \ph,0-\frac{1}{2}|\operatorname{det}(\alpha)|^{-1/2}f),
\end{equation}
and thus $f$ is uniquely determined via the condition that the second component has to be equal to $\delta \lambda$, i.e., 
\begin{equation}\label{eq:f}
 f=2|\operatorname{det}(\alpha)|^{1/2}\delta\lambda   
\end{equation} 

Thus, we have reduced the infimum to an infimum over all trace-free $\delta\beta\in T_{\alpha}\on{Gau}(TM)$. Next we show that the decomposition $\delta \alpha=\delta \beta+\frac1d f\alpha$ is orthogonal w.r.t. the metric $G^{\on{Aut}(TM)}$. Therefore we calculate for trace free $\delta \beta$ and arbitrary $f$, that
\begin{align}
&G^{\on{Aut}(TM)}_{(\ph,\alpha)}((0,\frac1d f\alpha),(0,\delta \beta))\\&\qquad= 
0+d\Lambda\int \operatorname{Tr}(g_0 \alpha^{-1} f\alpha g_0^{-1} \delta \beta^\top \alpha^{-\top})|\operatorname{det}(\alpha)|^{-1} \vol(g_0)\\&\qquad
= 
0+d\Lambda\int f \operatorname{Tr}(\delta \beta^\top \alpha^{-\top})|\operatorname{det}(\alpha)|^{-1} \vol(g_0)\\&\qquad
= d\Lambda\int f \operatorname{Tr}(\alpha^{-1}\delta \beta )|\operatorname{det}(\alpha)|^{-1} \vol(g_0)
=d\Lambda\int 0\vol(g_0) =0
\end{align}
Thus we see that the infimum in~\eqref{eq:riemsub_Aut-Aut} is attained by choosing $\delta\beta=0$ and we get
\begin{align}
&\inf_{\delta\bar \ph\in T_{\bar \ph}\operatorname{Aut}(TM)} \left\{G^{\operatorname{Aut}(TM)}_{\bar \ph}(\delta \bar \ph,\delta\bar\ph) \,; T_{\bar \ph}\pi_3.(\delta \bar \ph) = \delta \tilde\ph\right\}\\&\qquad\qquad=
G^{\operatorname{Aut}(TM)}_{\bar \ph}((\delta \ph, \tfrac{2}{d}|\operatorname{det}(\alpha)|^{1/2}(\delta\lambda)  \alpha ),(0,\tfrac{2}{d}|\operatorname{det}(\alpha)|^{1/2}(\delta\lambda)  \alpha))
\\&\qquad\qquad=
\int \left(g_0(\delta \ph,\delta \ph)|\operatorname{det}(\alpha)|^{-1} +d\Lambda\tfrac{4}{d}(\delta\lambda)^2  \right)\vol(g_0)
\end{align}
which is equal to the metric $G^{\operatorname{Aut}(\mathcal{C}(M))}$  as required. It remains to prove the independence on the choice of preimage, which is clear as the above calculation only depends on $\alpha$ via its determinant.

\paragraph{{\bf $\pi_4$ is a Riemannian submersion:}} As in the discussion of $\pi_2$ we see that $\pi_4$ is surjective. 
We shall use here that symmetry of any $g\in \Met(M)$ means $g=g^t:= g^T\o i: TM\to T^{**}M\to T^*M$.
\begin{align*}
\pi_4(\ph,\alpha) &= (\ph,\alpha_*g_0)=(\ph,\alpha^{-T}g_0\alpha^{-1}) =: (\ph,g)
\\
T_{(\ph,\al)}\pi_4(\de\ph,\de\al) &=  (\de\ph, -\al^{-T}.\de\al^T.\al^{-T}.g_0.\al^{-1} -\al^{-T}.g_0.\al^{-1}.\de\al.\al^{-1})
\\&
=  (\de\ph, -(\de\al.\al^{-1})^T.g -g.\de\al.\al^{-1}) 
\\&
= (\de\ph, -(g.\de\al.\al^{-1})^t -g.\de\al.\al^{-1})
\\
\pi_4^{-1}(\ph,g) &= \{(\ph,\al): \al^{-T}g_0\al^{-1} = g\}
\\
T_{(\ph,\al)}\pi_4^{-1}(\ph,g) &= \{(0,\de\al):  -\al^{-T}.\de\al^T.\al^{-T}.g_0.\al^{-1} -\al^{-T}.g_0.\al^{-1}.\de\al.\al^{-1}=0\}
\\
&= \{(0,\de\al):  \al^{-T}.\de\al^T.g + g.\de\al.\al^{-1}=0\}
\\
&= \{(0,\de\al):   (g.\de\al.\al^{-1})^t = -g.\de\al.\al^{-1}\}\,.
\end{align*}
Note that $(g.\de\al.\al^{-1})^t = -g.\de\al.\al^{-1}$ means that is skew symmetric as bilinear form on $TM$, equivalently, that 
$\de\al.\al^{-1}$ is skew symmetric with respect to $g$.

A tangent vector $(\de\ph,\de\be)\in T_{(\ph,\al)}\on{Aut}(TM)$ is perpendicular with respect to $G^{\on{Aut}(TM)}$ to the tangent space $T_{(\ph,\al)}\pi_4^{-1}(\ph,g) $ of the fiber over $(\ph,g)$ iff for all $\de\al$ with $g.\de\al.\al^{-1}$  skew symmetric we have
\begin{align*}
0&= G^{\on{Aut}(TM)}_{(\ph,\al)}((0, \delta \al),(\delta \ph, \delta \be)) = 
 \\& 
    = \int \left(g_0(0,\delta \ph) +\tfrac{d\Lambda}{4}\on{Tr}(g_0 \alpha^{-1} \delta \alpha g_0^{-1} \delta \be^\top \alpha^{-\top}) \right)|\on{det}(\alpha)|^{-1} \vol(g_0) 
\\&
    =  \tfrac{d\Lambda}{4}\int\on{Tr}(\alpha^{-\top}g_0 \alpha^{-1} \delta \alpha g_0^{-1} \delta \be^\top )\vol(g)  
\\&
  =  \tfrac{d\Lambda}{4}\int\on{Tr}(g. \delta \alpha. \al^{-1}.\al. g_0^{-1}\al^T\al^{-T} \delta \be^\top ) \vol(g)
\\&
  =  \tfrac{d\Lambda}{4}\int\on{Tr}(g. \delta \alpha. \al^{-1}.g^{-1} (\delta \be.\al^{-1})^\top ) \vol(g)
\end{align*} 
which implies that $g.\de\be.\al^{-1}$ is symmetric as bilinear form, or that $\de\be.\al^{-1}$ is $g$-symmetric, since symmetric and skew symmetric forms are orthogonal complements.
So we get a well-defined horizontal bundle 
\begin{align}
\on{Hor}_{(\ph,\al)} &= \{(\de\ph,\de\be): \de\be.\al^{-1} \text{ is } g\text{-symmetric}\}\,.
\end{align}
which is a complement to the vertical bundle. The restriction of the projection $T_{(\ph,\al)}\pi_4$ to the horizontal bundle is an isometry since  for 
$\pi_4(\ph,\al)= (\ph,g)$ where $g=\al^{-T}g_0\al^{-1})$, and for $T_{(\ph,\al)}\pi_4(\de\ph,\de\be) = (\de\ph, -(g.\de\be.\al^{-1})^t -g.\de\be.\al^{-1})$, where $(g.\de\be.\al^{-1})^t=g.\de\be.\al^{-1}$, we have
\begin{align*}
&G^{W\x E}_{(\ph,g)}(T_{(\ph,\al)}\pi_4 (\de\ph,\de\be), T_{(\ph,\al)}\pi_4 (\de\ph,\de\be)) =
\int_M g_0(\de\ph,\de\ph)\vol(g)
\\&\qquad
+ \tfrac{d\const}{4}\int_M \on{tr}(g^{-1} \big((g.\de\be.\al^{-1})^t +g.\de\be.\al^{-1}) g^{-1} ((g.\de\be.\al^{-1})^t +g.\de\be.\al^{-1})\big)\vol(g)
\\&\quad
= \int_M g_0(\de\ph,\de\ph)\vol(g)
+ {d\Lambda}\int_M \on{tr}(\de\be.\al^{-1}. \de\be.\al^{-1})\vol(g)
\end{align*}
and
\begin{align*}
&G^{\on{Aut}(TM)}_{(\ph,\al)}((\de\ph, \de\be),(\de\ph, \de\be)) =  
\\&\quad
= \int \left(g_0(\delta \ph,\delta \ph) +d\Lambda \on{Tr}(g_0 \alpha^{-1} \delta \be g_0^{-1} \delta \be^\top \alpha^{-\top}) \right)|\on{det}(\alpha)|^{-1} \vol(g_0)  
\\&\quad
= \int \left(g_0(\delta \ph,\delta \ph) +\tfrac{d\Lambda}{4}\on{Tr}(\de\be.\al^{-1}. \de\be.\al^{-1}) \right)\vol(g).
\end{align*}
Here we used that
\begin{align*}
&
\on{Tr}(g_0 \alpha^{-1} \delta \be g_0^{-1} \delta \be^\top \alpha^{-\top}) 
= \on{Tr}(\alpha^{-\top}.g_0 .\alpha^{-1}. \delta \be.\al^{-1}.\al. g_0^{-1}\al^T.\al^{-T}. \delta \be^\top )
\\&\quad
= \on{Tr}(g. \delta \be.\al^{-1}.g^{-1}. (\delta \be.\al^{-1})^\top )
= \on{Tr}(\delta \be.\al^{-1}.g^{-1}. (\delta \be.\al^{-1})^\top.g^t )
\\&\quad
= \on{Tr}(\delta \be.\al^{-1}.g^{-1}. (g.\delta \be.\al^{-1})^t)
= \on{Tr}(\delta \be.\al^{-1}.\delta \be.\al^{-1});.   
\end{align*} 
That $\pi_5$ is a Riemannian submersion follows directly from the abstract result on infimal in 
Proposition~\ref{prop:infimalconv}, where $\mathcal G = \Diff(M)$, $\mathcal M=\Met M$, $\ell(\ph,g)= \ph_*g$, $\de\ph= X\o \ph$, $x=g$, and
\begin{align*}
G^1(g)_\ph(\de\ph,\de\ph) &= \int_M g_0(X,X)\vol(g),
\\
G^2_g(\de g,\de g) &= \int_M \on{tr}(g^{-1}\de g g^{-1}\de g)\vol(g).
\end{align*}
This concludes the proof.

\paragraph{{\bf $\pi_5$ is a Riemannian submersion:}} 
This follows directly from Proposition~\ref{prop:infimalconv}. For the convenience of the reader we also present a direct proof. 
We first show, that it is a submersion:
\begin{align*}
 \pi_5(\ph,g) &= \ph_*g = T^*\ph^{-1}\o g\o T\ph^{-1} =:\bar g
\\
T_{(\ph,g)}\pi_5(\de\ph,h) &= -\ph_*\L_Xg + \ph_*h = \ph_*(-\L_X g +h)\quad\text{ where }X=T\ph^{-1}\o \de\ph
\\
&= -\L_Y\ph_*g + \ph_*h \quad\text{ where }Y=\de\ph\o \ph^{-1} = \ph_*X
\\
\pi_5^{-1}(\bar g) &= \{(\ph,g)\in \Diff(M)\x \Met(M): \ph_*g= \bar g\} 
\\&
= \{ (\ph,\ph^*\bar g): \ph\in\Diff(M)\}
\\
T_{(\ph,\ph^*\bar g)}\pi_5^{-1}(\bar g) &= \{(Y\o \ph, \ph^*\L_Y\bar g): Y\in \X(M)\}
\end{align*}
Note the appearance of the continuity equation in $T\pi_5$.

It only remains to show that 
$T_{(\ph,g)}\pi_5: T_\ph\Diff(M)\x T_g\Met(M) \to T_{\ph_*g}\Met(M)$ is a norm-quotient mapping independently of $\ph$.
For $g=\ph^*\bar g$ and $k = T_{(\ph,g)}\pi_5(\de\ph,h) = -\L_Y \ph_*g +\ph_*h = -\L_Y\bar g +\ph_*h \in T_{\bar g}\Met(M)$  we have
\begin{align*}
\| k\|^2_{G^{\on{WE}}_{\bar g}} &= G^{\on{WE}}_{\bar g}(k,k) = G^{\on{WE}}_{\ph_*g}(T_{(\ph,g)}\pi_5(\de\ph,h),T_{(\ph,g)}\pi_5(\de\ph,h)) 
\\&
=\inf\Big\{\int_M g_0(Y,Y) \vol(\bar g)+
d\const \int_M \on{tr}(\bar g^{-1}\bar h \bar g^{-1}\bar h)\on{vol}(\bar g): \\
&\qquad\qquad Y\in \X(M),\bar h\in T_{\bar g}\Met(M) \text{ with } k= -\mathcal L_{Y} \bar g + \bar h\Big\}
\\&
=\inf\Big\{\int_M g_0(Y\o\ph,Y\o\ph) \vol(g)+
d\const  \int_M \on{tr}(g^{-1} h g^{-1} h)\on{vol}(g): \\
&\qquad\qquad Y\in \X(M),h=\ph^*\bar h\in T_{g}\Met(M) \text{ with } k= -\mathcal L_{Y} \bar g + \bar h\Big\}
\end{align*}
where in the end we applied $\ph^*$ to both integrands. This should be equal to 
\begin{align*} 
&\inf \big\{\| (Y\o\ph,h)\|^2_{G^{W\x E}_{(\ph,g)}}: (Y\o \ph,h)\in  T_\ph\Diff(M) \x T_g\Met(M)
\\&
\qquad\qquad h \in T_g\Met(M), Y\in \X(M)\text{ with }T_{(\ph,\ph^*\bar g)}\pi_5(Y\o\ph,h) = k \big\}
\\&
= \inf \Big\{\int_M g_0(Y\o\ph,Y\o\ph)\vol(g) + \tfrac{d\const}{4}\int_M \on{tr}(g^{-1}hg^{-1}h)\vol(g) : 
\\&\qquad\qquad h \in T_g\Met(M), Y\in \X(M) \text{ with } -\L_Y\bar g +\ph_*h = k \Big\}
\end{align*}
independently of $\ph$, which is the case by the reparameterization invariance of both integrals.
Since squared norms are homogeneous of order 2, that is, $\|tk\|^2= t^2\|k\|^2$, it also follows that the induced norm on the image by such a Riemannian submersion actually comes from a Riemannian metric. Unfortunately, the horizontal lifts for these Riemannian submersions are not simple: They may only lie in a appropriate completion of each fiber. 

\paragraph{\bf $\pi_8$ is an isometry.}
For  $(\ph,\la)\in \on{Aut}(\mathcal C(M))$ with $\pi_8(\ph,\la)= (\ph, \la^2.\rh_0) = (\ph,\rh)$ 
we have
$T_{(\ph,\la)}\pi_8(\de\ph,\de\la) = (\de\ph,2\la.\de\la.\rh_0) = (\de\ph,\de\rh)$, so that  $\de\la=\frac{\de\rh}{2\la\rh_0}$. Thus 
\begin{multline}
G^{\on{Auf}(\mathcal C(M))}_{(\ph,\la)}\big((\de\ph,\de\la),(\de\ph,\de\la)\big) = \int\Big( \la^2 g_0(\de\ph,\de\ph) + \La(\de\la)^2\Big)\rh_0
\\
= \int_M g_0(\de\ph,\de\ph)\rh +\frac{ \La}{4} \int_M \Big(\frac{\de\rh}{\rh}\Big)^2\rh = G^{W\x FR}_{(\ph,\rh)}\big((\de\ph,\de\rh),(\de\ph,\de\rh)\big) 
\end{multline}

\paragraph{\bf The other maps.} $\pi_6$ is a Riemannian submersion by Lemma \ref{lem:RiemSub}, since $\pi_6\o \pi_4 = \pi_8 \o \pi_2$ and $\pi_4$ are Riemannian submersions. Similarly, $\pi_7$ is a Riemannian submersion, since $\pi_7\o \pi_8 = \pi_0$
and $\pi_8$ are  Riemannian submersion, again by Lemma \ref{lem:RiemSub}.
\end{proof}

\subsection*{The Geodesic Distance of the Wasserstein--Ebin Metric}
Finally, we will consider the induced geodesic distance function of the Wasserstein--Ebin metric, for which we obtain the following lower and upper bounds in terms of the geodesic distances of the Ebin and Wasserstein--Fisher--Rao metrics.
\begin{lem}\label{lem:geodesic_dist}
For $g_1$ and $g_2\in \Met(M)$
let $\operatorname{dist}^{\operatorname{WE}}(g_1,g_2)$ denote the geodesic distance of the Wasserstein--Ebin metric. Then  $\operatorname{dist}^{\operatorname{WE}}$ satisfies the bounds
\begin{equation}
\operatorname{dist}^{\operatorname{WFR}}(\vol(g_1),\vol(g_2))\leq \operatorname{dist}^{\operatorname{WE}}(g_1,g_2)\leq \frac{\sqrt{d\Lambda}}{2}\operatorname{dist}^{\operatorname{E}}(g_1,g_2), 
\end{equation}
where $\operatorname{dist}^{\operatorname{E}}$ denotes the geodesic distance of the Ebin metric on the space of Riemannian metrics and where $\operatorname{dist}^{\operatorname{WFR}}$ denotes the geodesic distance of the Wasserstein--Fisher--Rao distance on the space of smooth densities. 
\end{lem}
\begin{proof}
The lower bound follows directly from the Riemannian submersion result, cf. Theorem~\ref{thm:riemsub_WFR}, whereas the upper bound follows by choosing $h=\partial_t g$ and $v=0$. 
 \end{proof}

In the next lemma, we derive an additional pointer towards the non-degeneracy of the Wasserstein--Ebin metric. Therefore we fix a Riemannian metric $g_0$ with volume form $\mu$
and consider $\operatorname{Orb}^{\rh_0}(g_0)$, the orbit of the group of $\rho_0$-volume preserving diffeomorphisms through $g_0$ as a subset of the manifold of Riemannian metrics, i.e., 
consider 
\begin{equation}\label{eq:orbSDiff}
\operatorname{Orb}^{\rh_0}(g_0):=\left\{g=\varphi_*g_0: \varphi\in \Diff_{\rh_0}(M)\right\}\subset \Met_{\rh_0}(M)\subset \Met(M)  
\end{equation}
and equip it with the restriction of the Wasserstein--Ebin Riemannian metric, which corresponds to the Lagragian~\eqref{EqWE}, but where the infimum is calculated over all paths
$g:[0,1]\to \operatorname{Orb}^{\rh_0}(g_0)$, which restricts the set of admissable $h$ in the continuity equation $\partial_t g= -\mathcal L_v g +h$ to those of the form
$h=-\mathcal L_{\varphi_*w} g=\varphi_*(\mathcal L_{w}g_0)$ for some time-dependent vector field $w$.

In the following theorem, we will show the non-degeneracy of the corresponding geodesic distance on this subset of the space of all Riemannian metrics:
\begin{lem}\label{lem:geod_orbit}
Let $g_0\in \Met_{\rho_0}(M)$ and let $\operatorname{Orb}^{\rh_0}(g_0)$ be the orbit of $\Diff_{\rh_0}(M)$ as defined in~\eqref{eq:orbSDiff}. Then the geodesic distance of the infimal convolution of the restrictions of the Wasserstein and the Ebin metric to $\operatorname{Orb}^{\rh_0}(g_0)$ is non-degnerate. 
\end{lem}
\begin{proof}

We aim to use Theorem~\ref{lem:non-degenerate} with $G^1$ being the $L^2$-inner product on the vector field $v$ and $G^2$ being the Ebin-metric applied to $\varphi_*(\mathcal L_w g_0)$. We need to show that $G^1$ is uniformly upper bounded by $G^2$ on $G^1$ metric balls. 
First we observe that the $L^2$-inner product on the vector field 
$v$ really defines a Riemannian metric on the orbit, since any tangent vector in $T_{g}\on{Orb}^{\rho_0}(g_0)$ can be written as 
$\varphi_*(\mathcal L_v g_0)$, where $g=\varphi_*g_0$. Furthermore, the volume density of any metric in the orbit is equal to $\rho_0$ and thus we have for $g=\varphi_*g_0$ that
\begin{equation}
G^1_{g}(\varphi_*(\mathcal L_v g_0),\varphi_*(\mathcal L_v g_0))=\int_M g_0(v,v) \vol(g)=\int_M g_0(v,v) \rho_0,
\end{equation}
which implies that on the orbit the metric coincides with the right-invariant $L^2$-metric.

Next we aim to simplify the formula for $G^2$. We have
\begin{align}
&G^2_{g}(\varphi_*(\mathcal L_{w} g_0), \varphi_*(\mathcal L_{w} g_0))\\&=\int_M \on{tr}\Big((\varphi_*g_0)^{-1}\varphi_*(\mathcal L_{w}g_0)(\varphi_*g_0)^{-1}\varphi_*(\mathcal L_{w}g_0)\Big)\on{vol}(\varphi_*g_0)\\
&= \int_M \on{tr}\Big(g_0^{-1} \mathcal (\mathcal L_{w}g_0) g_0^{-1} \mathcal (\mathcal L_{w}g_0)\Big)\rho_0\;. 
\end{align}
Next we note, that $\mathcal L_{w}g_0=dw^t+dw$, where the transposed has to be calculated with respect to the metric $g_0$, i.e.,
the induced metric on the orbit is a first order, homogenous Sobolev metric and assuming that $g_0$ does not admit any killing fields we have that 
$G_1\lesssim G_2$ as required. Thus by 
Lemma~\ref{lem:non-degenerate} the geodesic distance of the infimal convolution of the pull-back of the Ebin-metric with the $L^2$-metric is non-degenerate, since the geodesic distance of the $L^2$-metric on $\Diff_{\mu}$ is non-degenerate.
\end{proof}

 \begin{question*}[Non-degeneracy of the Wasserstein--Ebin geodesic distance]
Note, that the lower bound in Lemma~\ref{lem:geodesic_dist} vanishes if and only if $g_1$ and $g_2$ induce the same volume form. Consequently the above lemma does not prove the non-degeneracy of the geodesic distance. It is tempting to believe that the non-degeneracy of the geodesic distance should follow directly from the non-degeneracy of the geodesic distances of the $L^2$-metric and the Ebin metric\footnote{The latter has been shown by Clarke in~\cite{clarke2010metric}.} as the Wasserstein--Ebin metric is defined as the infimal convolution of these two metrics. This is, however, not true as we have seen in Theorem~\ref{thm:vanishing}. Similarly, it seems plausible to conclude the non-degeneracy from the non-degeneracy of the metric on $\on{Aut}(TM)$ and the fact that $\pi_2$ is a Riemannian submersion. This argument is, however, flawed as this would require the projection $\pi_2$ to be continuous for the topologies induced by the respective metrics. This is not clear at all as the projection is of order one, but the metrics are only of order zero. Consequently, while we suspect that the geodesic distance of the Wasserstein--Ebin metric is indeed non-degenerate, at the moment we have to leave this as an open question for future research. 
\end{question*}

\section{Kullback-Leibler-Divergences on $\Met(M)$ and a Static Formulation of UORMT}\label{sec:static}
For unbalanced OMT it has been shown that there exists an equivalent static formulation using 
the Kullback-Leibler (KL) divergence as a soft constraint, see e.g.~\cite{chizat2018unbalanced}. It seems natural to define an analogous soft-constraint static version of UORMT, which we will attempt at the end of this section. Before doing that, we will introduce a generalization of the KL-divergence to the space of Riemannian metrics, which we believe is of interest in its own.  
\subsection*{A KL divergence on $\Met(M)$ inspired by matrix information geometry}
The following definition is motivated by the observation, that a Riemannian metric assigns to each point $x \in M$ a positive definite inner product on the tangent space $T_x M$. Equivalently, $g(x)$ may be viewed as a covariance matrix, defining a centered Gaussian distribution on $T_x M$. From this perspective, a Riemannian metric can be interpreted as a field of local Gaussian models over $M$. This observation suggests defining a divergence between two metrics $g_0$ and $g_1$ by comparing the associated Gaussian distributions point-wise and integrating over the manifold:
\begin{defn}[KL-divergence on $\Met(M)$]
 Given two Riemannian metrics $g_0,g_1\in \Met(M)$, we define the KL-type divergence via:
\begin{equation}\label{eq:KLdivergence}
\mathcal D^{\on{KL}-\Met}(g_0 \,\|\, g_1)
=
\frac{1}{2} \int_M
\left(
\operatorname{tr}\!\left(g_1^{-1} g_0\right)
- 2\log\left(\frac{\vol(g_0)}{\vol(g_1)}\right)
- d
\right)
\, \vol(g_1)\;,
\end{equation}  
where $d=\on{dim}(M)$.
\end{defn}
Note, that in finite-dimensional matrix information geometry, the integrand in the above expression is a well known construction, see e.g. \cite{NielsenBhatia2013,Bhatia2007,AmariNagaoka2000}.  Our construction extends the pointwise divergence from matrix geometry to the infinite-dimensional manifold of Riemannian metrics, where it also shares similarities with the so-called \emph{specific relative entropy}, see for instance \cite{benamou2024entropic}, which is a natural entropy on the space of diffusion processes.
 Furthermore, the construction admits a natural geometric interpretation: the term $\operatorname{tr}(g_1^{-1} g_0)$ measures the discrepancy between the quadratic forms defined by the metrics, while the logarithmic term compares the induced volume densities $\vol(g_0)$ and $\vol(g_1)$. Thus, the proposed divergence decomposes into contributions reflecting both \textit{shape distortion} and \textit{volume distortion}. 

In the following result we list some of its basic properties and will provide its relation to the Ebin metric:
\begin{thm}\label{thm:KLdivergence}
Let $M$ be a compact smooth manifold of dimension $d$, and let
$\mathcal D^{\on{KL}-\Met}$ be the KL-type divergence on $\Met(M)$ defined in~\eqref{eq:KLdivergence}. We have:
\begin{enumerate}
\item $\mathcal D^{\on{KL}-\Met} : \Met(M)\times \Met(M)\to \mathbb{R}$ is a divergence function in the sense of information geometry.

\item The Riemannian metric induced by $\mathcal D^{\on{KL}-\Met}$ is a multiple of the Ebin metric; more precisely,
\[
G_g(h,k)
=
-\left.
\frac{\partial^2}{\partial s\,\partial t}
\mathcal D^{\on{KL}-\Met}(g+sh\|g+tk)
\right|_{s=t=0}
=
\frac12
\int_M
\operatorname{tr}(g^{-1}h\,g^{-1}k)\,\vol(g).
\]

\item\label{item3} The divergence descends via the volume map
\[
\pi_1:\Met(M)\to \Dens(M),
\qquad
\pi_1(g)=\vol(g),
\]
to the $f$-divergence
\[
\widehat{\mathcal D}^{\on{KL}-\Dens}(\rho_0\|\rho_1)
:=
\inf_{\substack{g_0,g_1\in\Met(M)\\
\vol(g_0)=\rho_0,\ \vol(g_1)=\rho_1}}
\mathcal D^{\on{KL}-\Met}(g_0\|g_1),
\]
which is given by
\[
\widehat{\mathcal D}^{\on{KL}-\Dens}(\rho_0\|\rho_1)
=
\int_M
f_d\!\left(\frac{\rho_0}{\rho_1}\right)\rho_1,
\text{ with }
f_d(r)
=
\frac12\left(d r^{2/d}-2\log r-d\right).
\]
For $d=2$, this becomes
\[
\widehat{\mathcal D}^{\on{KL}-\Dens}(\rho_0\|\rho_1)
=
\int_M
\left(
r-\log r-1
\right)\rho_1,
\qquad
r=\frac{\rho_0}{\rho_1},
\]
which is the $f$-divergence associated with the Burg/Itakura--Saito generator. \end{enumerate}
\end{thm}
\begin{rem}
In contrast to the classical Kullback--Leibler divergence on probability densities, which induces a dually flat information-geometric structure~\cite{AmariNagaoka2000,lenells2014amari,bauer2025riemannian}, the divergence $\mathcal D^{\on{KL}-\Met}$ does not, in general, give rise to a dually flat structure on $\mathrm{Met}(M)$. The obstruction in our situation, is that the reference measure $\vol(g_1)$ depends on the second argument, so the divergence is not generated by a global convex potential. A natural idea is to restrict to the subspace
\[
\mathrm{Met}_\nu(M)
=
\{g \in \mathrm{Met}(M) : \vol(g)=\nu\}
\]
for a fixed volume density $\nu$. On this subspace, the logarithmic term vanishes and the divergence reduces to
\[
\mathcal D^{\on{KL}-\Met}(g_0 \| g_1)
=
\frac{1}{2} \int_M
\left(
\operatorname{tr}(g_1^{-1} g_0)
-
d
\right)
\, \nu.
\]
However, $\mathrm{Met}_\nu(M)$ is not an affine subspace with respect to either the coordinates $g$ or $g^{-1}$, and consequently this restriction does not, in general, inherit a dually flat structure. This contrasts with the classical setting of probability densities, where the normalization constraint defines an affine subspace and the KL divergence retains its Bregman structure.
\end{rem}

\begin{proof}
We first show that $\mathcal D^{\on{KL}-\Met}$ is a true divergence function in the sense of information geometry~\cite{AmariNagaoka2000}, i.e., 
that it is nonnegative and vanishes if and only if its two arguments coincide, and that it is sufficiently smooth with vanishing first derivatives on the diagonal.

Therefore, let
\(
A:=g_1^{-1}g_0\in \Gamma(\on{End}(TM)).
\)
Then \(A_x\) is positive definite for every \(x\in M\) with respect to the inner product $(g_1)_x$, and
\[
\det(A)
=
\det(g_1^{-1}g_0)
=
\left(\frac{\vol(g_0)}{\vol(g_1)}\right)^2.
\]
Thus the integrand of \(\mathcal D^{\on{KL}-\Met}\) can be written as
\(
\frac12
\left(
\on{tr}(A)-\log\det(A)-d
\right).
\)
Let \(\lambda_1,\dots,\lambda_d>0\) be the eigenvalues of \(A_x\). Then
\[
\on{tr}(A)-\log\det(A)-d
=
\sum_{i=1}^d
\left(\lambda_i-\log\lambda_i-1\right).
\]
Since for all $s\geq 0$ we have that
\(
s-\log s-1\ge 0
\)
with equality if and only if \(s=1\), it follows that
\(
\mathcal D^{\on{KL}-\Met}(g_0\|g_1)\ge 0.
\)
Equality holds if and only if all \(\lambda_i=1\) pointwise, hence if and only if
\(
A
\) is the identity matrix
or equivalently \(g_0=g_1\). Smoothness follows from the smoothness of the maps
\[
(g_0,g_1)\mapsto g_1^{-1}g_0,
\qquad
A\mapsto \on{tr}(A),
\qquad
A\mapsto \log\det(A),
\qquad
g_1\mapsto \vol(g_1),
\]
on the open cone of positive definite metrics. This proves the nonnegativity, definiteness, and smoothness of $\mathcal D^{\on{KL}-\Met}$. 

To verify that it defines a divergence function it only remains to check that the first derivatives vanish along the diagonal. Therefore let \(g\in \Met(M)\), \(h\in T_g\Met(M)\) and consider a variation
\(
g_s=g+sh.
\)
Using that 
\begin{align*}
\frac{d}{ds}\left.(g^{-1}g_s)\right|_{s=0}=
\frac{d}{ds}\left.(\on{id}+sg^{-1}h)\right|_{s=0}=g^{-1}h
\end{align*}
we obtain
\begin{align*}
\frac{d}{ds}\left.
\mathcal D^{\on{KL}-\Met}(g_s\|g)
\right|_{s=0}=
\frac{1}{2} \frac{d}{ds}\left.\int_M
\left(
\operatorname{tr}\!\left(g^{-1} g_s\right)
- 2\log\left(\frac{\vol(g_s)}{\vol(g)}\right)
- d
\right)
\, \vol(g)\right|_{s=0}\\
\frac{1}{2} \int_M
\left(
\operatorname{tr}\!\left(g^{-1}h\right)
- \operatorname{tr}\!\left(g^{-1}h\right)
\right)
\, \vol(g)=0,
\end{align*}
where we used that $d(\log\det(A))(B)=\on{tr}(A^{-1}B)$. 
By a similar calculation we obtain that \[\left.
\frac{d}{ds}
\mathcal D^{\on{KL}-\Met}(g\|g_s)
\right|_{t=0}
=
0,
\] which concludes the proof that \(\mathcal D^{\on{KL}-\Met}\) is a divergence function in the sense of information geometry.

To compute the induced Riemannian metric we show by direct calculation that 
\[
\left.
\partial_s\partial_t
\mathcal D^{\on{KL}-\Met}(g+sh\|g+tk)
\right|_{s=t=0}
=
-\frac12
\int_M
\on{tr}(g^{-1}h\,g^{-1}k)\,\vol(g).
\]
Thus the Riemannian metric induced by the divergence is
\[
G_g(h,k)
:=
-
\left.
\partial_s\partial_t
\mathcal D^{\on{KL}-\Met}(g+sh\|g+tk)
\right|_{s=t=0}
=
\frac12
\int_M
\on{tr}(g^{-1}h\,g^{-1}k)\,\vol(g),
\]
which is exactly one half of the Ebin metric.

It remains to compute the divergence induced on densities by the volume map. Let
\[
\rho_0,\rho_1\in \Dens(M),
\qquad
r:=\frac{\rho_0}{\rho_1}.
\]
Define
\[
\mathcal D^{\on{KL}-\Dens}(\rho_0\|\rho_1)
:=
\inf_{\substack{g_0,g_1\in\Met(M)\\
\vol(g_0)=\rho_0,\ \vol(g_1)=\rho_1}}
\mathcal D^{\on{KL}-\Met}(g_0\|g_1).
\]
For fixed \(g_0,g_1\), set again
\(
A=g_1^{-1}g_0.
\)
The constraint \(\vol(g_0)=\rho_0\), \(\vol(g_1)=\rho_1\) implies
\(
\det(A)=r^2.
\)
Since the term \(\log\det(A)\) is fixed under this constraint, the pointwise minimization reduces to minimizing
\(
\operatorname{tr}(A)
\)
over $g_1$-positive definite endomorphisms \(A\) with determinant \(r^2\). Let
\(
\lambda_1,\dots,\lambda_d>0
\)
be the eigenvalues of \(A\). Then
\[
\prod_{i=1}^d \lambda_i=r^2,
\qquad
\on{tr}(A)=\sum_{i=1}^d\lambda_i.
\]
By the arithmetic--geometric mean inequality,
\[
\frac1d\sum_{i=1}^d\lambda_i
\ge
\left(\prod_{i=1}^d\lambda_i\right)^{1/d}
=
r^{2/d}.
\]
Hence
\[
\on{tr}(A)\ge d\,r^{2/d},
\]
with equality if and only if
\[
\lambda_1=\cdots=\lambda_d=r^{2/d}.
\]
Equivalently,
\(
A=r^{2/d}\Id,
\)
or
\(
g_0=r^{2/d}g_1.
\)
Thus the infimum over all metrics with prescribed volume densities is attained pointwise precisely on conformal pairs satisfying
\[
g_0=
\left(\frac{\rho_0}{\rho_1}\right)^{2/d}g_1.
\]
Substituting this minimizer into the divergence gives
\[
\on{tr}(A)=d\,r^{2/d},
\qquad
\log\det(A)=\log(r^2)=2\log r.
\]
Therefore
\[
\mathcal D^{\on{KL}-\Dens}(\rho_0\|\rho_1)
=
\frac12
\int_M
\left[
d\,r^{2/d}
-
2\log r
-
d
\right]\rho_1.
\]
Equivalently,
\[
\mathcal D^{\on{KL}-\Dens}(\rho_0\|\rho_1)
=
\int_M
f_d\!\left(\frac{\rho_0}{\rho_1}\right)\rho_1,
\]
where
\[
f_d(r)
=
\frac12
\left(
d\,r^{2/d}
-
2\log r
-
d
\right).
\]
Thus the induced divergence on densities is an \(f\)-divergence.
\end{proof}

Recall that the Ebin metric projects (up to a constant factor, and exactly in dimension $d=2$) onto the Fisher--Rao metric under the map
\[
g \mapsto \vol(g).
\]
\begin{rem}[Relation to the Itakura--Saito divergence]
As we have seen in point~\eqref{item3} of Theorem~\ref{thm:KLdivergence}, the divergence $\mathcal D^{\on{KL}-\Met}$ induces an $f$-divergence on the space of densities. This is, however, not the standard Kullback--Leibler divergence of information geometry, which corresponds to the generator
\[
f(r)=r\log r.
\]
Instead, in dimension $d=2$, the induced divergence is the $f$-divergence associated with the generator
\[
f(r)=r-\log(r)-1,
\]
commonly known as the Burg, or Itakura--Saito, divergence~\cite{AmariNagaoka2000}.
\end{rem}

\subsection*{A KL divergence on $\Met$ that projects onto the standard KL divergence.}
In this part we will construct an alternative KL divergence on $\Met(M)$. While the definition feels slightly more ad-hoc, it will have the advantage that it induces the standard KL divergence on $\Dens(M)$:
\begin{thm}
Let $M$ be a compact smooth manifold of dimension $d$.
Let
\[
\widetilde{\mathcal D}^{\on{KL}-\Met}(g_0\|g_1)
:=
\frac{2}{d}\,
{\mathcal D}^{\on{KL}}(\vol(g_0)\|\vol(g_1))
+
\mathcal D^{\on{shape}}(g_0\|g_1),
\]
where ${\mathcal D}^{\on{KL}}$ is the standard KL divergence on $\Dens(M)$ given by
\[
{\mathcal D}^{\on{KL}}(\rho_0\|\rho_1)
=
\int_M \log(\tfrac{\rho_0}{\rho_1})\rho_0 + \int_M \rho_1 -\int_M\rho_0
\]
and where $\mathcal D^{\on{shape}}$ is defined via

\[
\mathcal D^{\on{shape}}(g_0\|g_1)
:=
\frac12\int_M
\left(
\operatorname{tr}\left(g_1^{-1}g_0\right)-d\left(\frac{\vol(g_0)}{\vol(g_1)}\right)^{2/d}
\right)\vol(g_1).
\]

Then:
\begin{enumerate}
\item $\widetilde{\mathcal D}^{\on{KL}-\Met}: \Met(M)\times \Met(M)\to \mathbb{R}$ is a divergence function in the sense of information geometry.

\item The Riemannian metric induced by $\widetilde{\mathcal D}^{\on{KL}-\Met}$ is (one half times) the Ebin metric.

\item Under the volume map, $\widetilde{\mathcal D}^{\on{KL}-\Met}$ descends to a multiple of the standard KL divergence on $\Dens(M)$. 
\end{enumerate}
\end{thm}

\begin{proof}
To prove nonnegativity we first note that the KL term satisfies
\(
{\mathcal D}^{\on{KL}}(\vol(g_0)\|\vol(g_1))\ge 0,
\)
with equality if and only if \(\vol(g_0)=\vol(g_1)\). To prove the nonnegativity of the second term $\mathcal D^{\on{shape}}$ let
\[
A:=g_1^{-1}g_0,
\qquad
r:=\frac{\vol(g_0)}{\vol(g_1)},
\qquad
\widetilde A:=r^{-2/d}A.
\]
Then
\(
\det(A)=r^2\), and \(
\det(\widetilde A)
=
r^{-2}\det(A)
=
1.
\)
Moreover, since
\(\widetilde A\) is positive definite with determinant \(1\), if
\(\lambda_1,\dots,\lambda_d>0\) are its eigenvalues, then
\(
\prod_{i=1}^d \lambda_i=1.
\)
By the arithmetic--geometric mean inequality,
\[
\operatorname{tr}(\widetilde A)
=
\sum_{i=1}^d\lambda_i
\ge
d,
\]
with equality if and only if
\(
\lambda_1=\cdots=\lambda_d=1.
\)
Thus
\(
\operatorname{tr}(\widetilde A)-d\ge 0,
\)
with equality if and only if
\(
\widetilde A=\Id.
\)
Thus we obtain the desired nonnegativity by writing 
\begin{equation}\label{eq:dshape}
\mathcal D^{\on{shape}}(g_0\|g_1)
=\frac12\int_M
 \left(
 \operatorname{tr}(A)-d r^{2/d}
 \right)\vol(g_1)=
\frac12\int_M
 r^{2/d}
 \left(
 \operatorname{tr}(\widetilde A)-d
 \right)\vol(g_1)\geq 0.
 \end{equation}
Consequently,
\(
\widetilde{\mathcal D}^{\on{KL}-\Met}(g_0\|g_1)\ge 0.
\)  
The sum of two positive term vanishes if and only if both terms vanish. Hence
we have in particular that \(
\vol(g_0)=\vol(g_1),
\), i.e., $r=1$. Here we used again  
that KL is a divergence on the space of densities. As discussed above, vanishing of the  term $\mathcal D^{\on{shape}}$ implies that \(\widetilde A=\Id.
\) and thus also that $A=g_1^{-1}g_0=\Id$, since $r=1$. 
This implies that \(g_0=g_1\). Conversely, if \(g_0=g_1\), both terms clearly vanish. Hence
\[
\widetilde{\mathcal D}^{\on{KL}-\Met}(g_0\|g_1)=0
\quad\Longleftrightarrow\quad
g_0=g_1.
\]
Smoothness follows from the smoothness of the maps
\[
g\mapsto \vol(g),
\qquad
(g_0,g_1)\mapsto g_1^{-1}g_0,
\qquad
A\mapsto \operatorname{tr}(A),
\]
on the open cone of positive definite metrics. Thus
\(\widetilde{\mathcal D}^{\on{KL}-\Met}\) is a divergence function.

Next we compute the induced metric. Fix \(g\in\Met(M)\) and decompose tangent vectors $h,k$ into its trace free and pure trace components via
\[
h=h_0+\varphi g,
\qquad
k=k_0+\psi g,\qquad
\operatorname{tr}_g(h_0)=0,
\qquad
\operatorname{tr}_g(k_0)=0.
\]
where 
\(
\varphi=\frac1d\operatorname{tr}_g(h)\) and 
\(
\psi=\frac1d\operatorname{tr}_g(k).
\)
Using this trace decomposition, we obtain the orthogonal splitting
\[
\operatorname{tr}(g^{-1}h\,g^{-1}k)
=
\operatorname{tr}(g^{-1}h_0\,g^{-1}k_0)
+
d\,\varphi\psi.
\]
and thus we can rewrite the Ebin metric as
\[
G^{\on{Ebin}}_g(h,k)
=
\frac12\int_M
\operatorname{tr}(g^{-1}h_0\,g^{-1}k_0)\,\vol(g)
+
\frac d2\int_M \varphi\psi\,\vol(g).
\]
We now show that the second variation of
\(\widetilde{\mathcal D}^{\on{KL}-\Met}\) gives exactly this expression. First, consider the volume part. Since
\[
d\,\vol_g(h)
=
\frac12\operatorname{tr}_g(h)\,\vol(g)
=
\frac d2 \varphi\,\vol(g),
\]
the Fisher--Rao metric induced by the classical KL divergence satisfies
\[
G^{\on{FR}}_{\vol(g)}
\bigl(d\vol_g(h),d\vol_g(k)\bigr)
=
\int_M
\left(\frac d2\varphi\right)
\left(\frac d2\psi\right)
\vol(g)
=
\frac{d^2}{4}
\int_M \varphi\psi\,\vol(g).
\]
Multiplying KL by \(2/d\), its Hessian contribution becomes
\[
\frac{2}{d}\cdot
\frac{d^2}{4}
\int_M \varphi\psi\,\vol(g)
=
\frac d2
\int_M \varphi\psi\,\vol(g).
\]
This is exactly the conformal part of the Ebin metric.

It remains to compute the shape part for which we will use the formula in~\eqref{eq:dshape}. At the diagonal \(g_0=g_1=g\), we have
\[
r=1,
\qquad
A=\Id,
\qquad
\widetilde A=\Id.
\]
Since
\(
\operatorname{tr}(\widetilde A)-d=0
\)
at the diagonal, and since its first variation also vanishes at the diagonal, the derivatives of the prefactor
\(r^{2/d}\) and of the density \(\vol(g_1)\) do not contribute to the mixed second variation. Thus it suffices
to compute the mixed second variation of
\(
\frac12\operatorname{tr}(\widetilde A)
\)
at the diagonal. Therefore consider again variations
\[
g_s=g+sh,
\qquad
g_t=g+tk.
\]
Calculating the first variations of the components of the shape part gives
\[
\left.\partial_s (g_t^{-1}g_s)\right|_{s=t=0}
=
g^{-1}h,
\qquad
\left.\partial_t (g_t^{-1}g_s)\right|_{s=t=0}
=
-\,g^{-1}k,
\]
and
\[
\left.\partial_s \left(\frac{\vol(g_s)}{\vol(g_t)}\right)\right|_{s=t=0}
=
\frac12\operatorname{tr}_g(h)
=
\frac d2\varphi,
\qquad
\left.\partial_t \left(\frac{\vol(g_s)}{\vol(g_t)}\right)\right|_{s=t=0}
=
-\frac12\operatorname{tr}_g(k)
=
-\frac d2\psi.
\]
Hence
\[
\left.\partial_s \left(\left(\frac{\vol(g_s)}{\vol(g_t)}\right)^{-2/d}g_t^{-1}g_s\right)\right|_{s=t=0}
=
g^{-1}h-\varphi\Id
=
g^{-1}h_0,
\]
and
\[
\left.\partial_t \left(\left(\frac{\vol(g_s)}{\vol(g_t)}\right)^{-2/d}g_t^{-1}g_s\right)\right|_{s=t=0}
=
-\,g^{-1}k+\psi\Id
=
-\,g^{-1}k_0.
\]
Using that the function
\(
B\mapsto \operatorname{tr}(B)-d
\)
restricted to \(\det B=1\) has a minimum at \(B=\Id\), we obtain that its second variation at \(\Id\) in trace-free directions \(U,V\) is given by
\(
\operatorname{tr}(UV).
\)
Therefore,
\[
-\left.
\partial_s\partial_t
\mathcal D^{\on{shape}}(g_s\|g_t)
\right|_{s=t=0}
=
\frac12\int_M
\operatorname{tr}(g^{-1}h_0\,g^{-1}k_0)\,\vol(g).
\]
This is exactly the trace-free part of the Ebin metric.
Combining the volume and shape contributions then shows

that the induced metric is the Ebin metric.

Finally, we compute the projection to densities. Let
\(
\rho_0,\rho_1\in\Dens(M)\) and 
consider the induced density divergence given by
\[
\widetilde{\mathcal D}^{\on{KL}-\Dens}(\rho_0\|\rho_1)
:=
\inf_{\substack{g_0,g_1\in\Met(M)\\
\vol(g_0)=\rho_0,\ \vol(g_1)=\rho_1}}
\widetilde{\mathcal D}^{\on{KL}-\Met}(g_0\|g_1).
\]
The first term,
\(
\frac{2}{d}{\mathcal D}^{\on{KL}}(\vol(g_0)\|\vol(g_1)),
\)
is already fixed by the constraints and equals
\(
\frac{2}{d}{\mathcal D}^{\on{KL}}(\rho_0\|\rho_1).
\)
As shown above, the shape term is nonnegative and vanishes precisely when
\(
\widetilde A=\Id,
\)
that is, when $g_0=r^{2/d}g_1$. 
Such pairs of metrics exist for any prescribed positive densities \(\rho_0,\rho_1\) and  therefore the infimum of the shape term is \(0\). Consequently
\[
\widetilde{\mathcal D}^{\on{KL}-\Dens}(\rho_0\|\rho_1)
=
\frac{2}{d}{\mathcal D}^{\on{KL}}(\rho_0\|\rho_1),
\]
which concludes the proof.
\end{proof}

\subsection*{Towards a Static Formulation of UORMT}
Using the notion of a divergence function we are now ready to define a static formulation of UORMT:
\begin{defn}[Static Unbalanced Optimal Riemannian Metric Transport]
Let $\mathcal D^{\Met}$ be a divergence function on the space $\Met(M)$ of all Riemannian metrics and let $\lambda\in \mathbb R_{>0}$ be a balancing parameter. Given two Riemannian metrics $g_0,g_1\in \Met(M)$ the static UORMT problem is given by
\begin{align}\label{eq:generalSUMT}
\inf_{\substack{\bar g_0\in \Met(M)\\\bar g_1\in \on{Orb}(\bar g_0)}}\left(\lambda \mathcal D^{\Met}(g_0,\bar g_0)+\operatorname{dist}^{\on{Orb}(\bar g_0)}(\bar g_0,\bar g_1)+ \lambda \mathcal D^{\Met}(\bar g_1,g_1)\right),
\end{align}
where $\operatorname{dist}^{\on{Orb}(\bar g_0)}$ denotes the induced Wasserstein-type distance on the orbit through $\bar g_0$.  
\end{defn}
Using the considerations of section~\ref{sec:ORMT} we can write the static UORMT problem as follows:
\begin{cor}
 Given two Riemannian metrics $g_0,g_1\in \Met(M)$ the static UORMT problem~\eqref{eq:generalSUMT} is equivalent to:
\begin{align}
\inf_{\substack{\bar g_0\in \Met(M)\\\varphi\in \Diff(M)}}\left(\lambda \mathcal D^{\Met}(g_0,\bar g_0)+\operatorname{dist}^{\Diff}(\operatorname{id},\varphi)+ \lambda \mathcal D^{\Met}(\varphi_* \bar g_0, g_1)\right),
\end{align}
where $\operatorname{dist}^{\Diff}$ denotes the $L^2$-distance on $\Diff(M)$.  
\end{cor}

\begin{question*}
There are several interesting questions arising from the above definition of a static version of UORMT: first, it seems natural to study under which condition on the divergence function $\mathcal D^{\Met}$ the static formulation defines a distance function on the space of all Riemannian metrics. Note, that symmetry of the formulation is obvious, but that the non-degeneracy would require one to assume certain properties of the divergence function.

A more challenging question concerns the link between the static and dynamic formulations. For the WFR metric in UOMT it has been shown that there exists an equivalent static formulation if one chooses the KL-divergence as divergence function on the space of densities. 
The proof of the equivalence of the two formulations relies heavily on the convexity of the UOMT problem. Unless  the UORMT problem has hidden-convexity properties such as in \cite{brenier2020examples}, one would expect at least an inequality between the two models (or a variant), as obtained in \cite{sejourne2021unbalancedGW}.
\end{question*}

\section{The Euler--$\alpha$ Lagrangian from the Wasserstein--Ebin fiber}\label{sec:Euleralpha}

Motivated by the connection between the Wasserstein--Fisher--Rao metric and the Camassa--Holm equation \cite{GALLOUET20184199}, we discuss a formal connection with fluid dynamics. In fact, \cite{GALLOUET20184199} extends the well known connection between incompressible Euler and optimal transport put forward by Brenier \cite{brenier1989least}.
Here, the point is that the Wasserstein--Ebin metric contains, on a
natural isotropy fiber, a similar metric tensor that appears in the averaged
Euler, or Euler--$\alpha$ geometry.
Recall the dynamic formulation of the Wasserstein--Ebin metric on
\(\operatorname{Met}(M)\).  At a metric \(g\), a tangent vector \(\delta g\)
is decomposed as
\[
  \delta g=-\mathcal L_v g+h,
\]
and the squared norm is given by the infimal convolution
\begin{equation*} G^{\operatorname{WE}}_g(\delta g,\delta g)
  =
  \inf_{\delta g=-\mathcal L_vg+h}
  \left\{
  \int_M |v|_g^2\,d\operatorname{vol}_g
  +
  \frac14
  \int_M
  \operatorname{tr}(g^{-1}hg^{-1}h)\,
  d\operatorname{vol}_g
  \right\}.
\end{equation*}
Then, fix a background metric $g_0$.  Consider the semi-direct product action
$$
 \bigl(\operatorname{Diff}(M)\ltimes\Omega^2_{\mathrm{sym}}(M)\bigr)
  \times \operatorname{Met}(M)
  \longrightarrow
  \Omega^2_{\mathrm{sym}}(M)
$$
defined by
$$
  (\varphi,h)\cdot g_0
  =
  \varphi_*(g_0+h)\,.
$$
The isotropy subgroup is defined by
$$
  \operatorname{Isot}(g_0)
  =
  \{(\varphi,h):\varphi_*(g_0+h)=g_0\}\,.
$$
At the identity \((\operatorname{id},0)\), a tangent vector to
\(\operatorname{Diff}(M)\ltimes\Omega^2_{\mathrm{sym}}(M)\) is a pair
\((v,k)\). Since
$$
  \left.\frac{d}{dt}\right|_{t=0}(\varphi_t)_*g_0
  =
  -\mathcal L_v g_0\,,
$$
the infinitesimal isotropy condition is
$  k=\mathcal L_vg_0$.
Therefore, the restriction of the Wasserstein--Ebin quadratic form to the
isotropy direction is
\begin{equation}
  \mathcal L_{g_0}(v)
  =
  \int_M |v|_{g_0}^2\,d\operatorname{vol}_{g_0}
  +
  \frac14
  \int_M
  \operatorname{tr}
  \bigl(
    g_0^{-1}(\mathcal L_vg_0)
    g_0^{-1}(\mathcal L_vg_0)
  \bigr)
  d\operatorname{vol}_{g_0}\,.
\end{equation}
Introducing the deformation tensor
$
  \operatorname{Def}_{g_0}(v)_{ij}
  =
  \frac12(\nabla_i v_j+\nabla_j v_i) = \frac 12 L_vg_0 \,,
$
the metric reads
\begin{equation}
  \mathcal L_{g_0}(v)
  =
  \int_M |v|_{g_0}^2\,d\operatorname{vol}_{g_0}
  +
  \int_M |\operatorname{Def}_{g_0}(v)|_{g_0}^2\,
  d\operatorname{vol}_{g_0}.
\end{equation}
Up to a multiplicative factor $\alpha^2$, this is the weak
right-invariant \(H^1\)-type Lagrangian that appears in the averaged Euler, or
Euler--\(\alpha\), equations.
Equivalently, one can rewrite the Ebin part as
\begin{equation}
  \frac14 \int_M \operatorname{tr}\bigl(g_0^{-1}(\mathcal L_vg_0) g_0^{-1}(\mathcal L_vg_0) \bigr)d\operatorname{vol}_{g_0} =\int_M\left[\frac12|d v^\flat|^2+|\delta v^\flat|^2 \operatorname{Ric}_{g_0}(v,v) \right] d\operatorname{vol}_{g_0}\,.
\end{equation}
Hence
\begin{equation}
  \mathcal L_{g_0}(v)
  =
  \int_M
  \left[
    |v|^2
    +
    \frac12 |d v^\flat|^2
    +
    |\delta v^\flat|^2
    -
    \operatorname{Ric}_{g_0}(v,v)
  \right]
  d\operatorname{vol}_{g_0}\,.
\end{equation}
In the  Euclidean case, this reduces to
    $$\mathcal L_{\mathrm{flat}}(v)=\int_{\mathbb R^d}\left( |v|^2+|\operatorname{Def} v|^2\right)\,dx\,.
$$
If, in addition, one imposes the incompressibility constraint $\operatorname{div}v=0$ that can be obtained as another isotropy constraint, this is the usual averaged Euler, or Euler--$\alpha$ Lagrangian.

\begin{question*}
For the incompressible Euler equation or for the Camassa-Holm equation, this connection was used to derive results on minimizing geodesics, \cite{brenier1989least,gallouet2020generalized}. These results are strongly based on the convexity obtained from the optimal transport formulation, which might be hidden for the Wasserstein--Ebin metric. Is it possible to derive properties on minimizing geodesics for $\alpha$-Euler equation, without relying on convexity as in \cite{brenier1989least,gallouet2020generalized}?
\end{question*}

\appendix
\section{The Lie group Structure of $\on{Aut}(TM)$}
In this section will discuss the structure of the semi-direct product $\on{Aut}(TN)$ in more details. We start by calculating the right trivialization in this group:
    \begin{equation}\label{eq:righttriv}
      (v, A) \coloneqq  (\delta \varphi,\delta \alpha) (\varphi,\alpha)^{-1} = (\delta \varphi \circ \varphi^{-1}, T\varphi. \delta \alpha. \alpha^{-1}. T\varphi^{-1})\,.
    \end{equation}
For the Lie algebra of $\on{Aut}(TM)$ recall that the second tangent bundle carries two vector bundle structures over $TM$ as follows, where the diagram on the right hand side is in charts derived from a chart on $M$: 
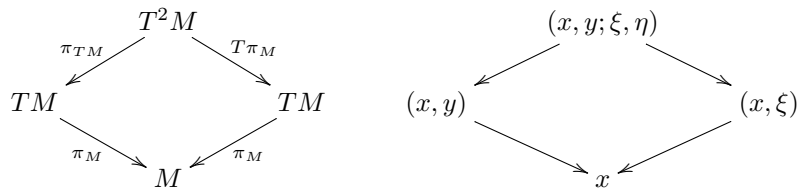
$$
\xymatrix@R5mm{
& T^2M \ar[dl]_{\pi_{TM}} \ar[dr]^{T\pi_M} &    & & (x,y;\xi,\et) \ar[dl] \ar[dr] &\\
TM \ar[dr]_{\pi_M} && TM \ar[dl]^{\pi_M}         & (x,y) \ar[dr] && (x,\xi) \ar[dl]  \\
& M &                                            & & x &
}$$
The Lie algebra $T_{\on{Id}}\on{Aut}(TM)$ of $\on{Aut}(TM)$ consists of all vector fields  
$TM\to T^2M$ which are $\pi_M-T\pi_M$-linear; locally $(x,y)\mapsto (x,y;0,B_x(y))$ for $B_x(y)$ linear in $y$. A tangent vector along $\al\in \on{Aut}(TM)$  locally looks like 
$(x,y)\mapsto (\bar\al(x),A_x(y);X(x),B_x(y))$ for $A_x(y)$, and $B_x(y)$ both linear in $y$.
The tangent mapping of $T:\Diff(M)\to \on{Aut}(TM)$ is given by $T_f T.X = \ka_M\o TX$ for $X$ a vetor field along $f$ where $\ka_M:T^2M\to T^2M$ is the canonical flip; locally we have 
$\ka_M(x,y;\xi;\et) = (x,\xi, y,\et)$ and 
\begin{align}
(T_f T.X)(x,y) &= (\ka\o TX)(x,y) = \ka(f(x),X(x); df(x)y, dX(x)y) 
\\&= (f(x), df(x)y; X(x), dX(x)y)    
\end{align}
The Lie algebra $T_{\on{Id}}\on{Gau}(TM)$ of $\on{Gau}(TM)$ is obviously $\Ga(M\leftarrow L(TM,TM))=\Ga(\on{End}(TM))$. It is embedded into $T_{\on{Id}}\on{Aut}(TM)$ by the vertical lift, locally
$$
(\on{emb}A)(x,y) = (x, A_x(y); 0, A_x(y)).
$$ 
The infinitesimal action of the Lie algebra $\X(M)$ of $\Diff(M)$ on the Lie algebra $\Ga(\on{End}(TM))$ of $\on{Gau}(TM)$ is via the Lie derivative. Namely, let 
$f:\mathbb R\to \Diff(M)$ be smooth with $f(0)=\on{Id}$ and $\p_t|_0 f(t) = X\in \X(M)$, then we get
$$
\p_t|_0 (Tf(t)^{-1}.\xi.Tf(t)) = \p_t|_0 f(t)^*\xi = \L_X\xi.
$$
Finally we shall need the right translation on $T\on{Aut}(TM)$ in terms of the semidirect product structure \eqref{Eqsemidirectright}. This is given  in general for $X: M\to TM$ as:
\begin{align}
    T\mu^{(f,\al)}(X,\xi) &= (X\o f, Tf^{-1}.\xi.Tf.\al)
    \\
    T\mu^{(f,\al)^{-1}}(X,\xi) &= (X\o f^{-1}, Tf.\xi.\al^{-1}.Tf^{-1})    
\end{align}

Using this right-trivialization it is easy to construct Riemannian metrics 
 on $\operatorname{Aut}(TM)$ that descend to metrics on the space of Riemannian metric. In the following we let $g = (\ph,\al)_* g_0$ and consider the Riemannian metric
\begin{equation}\label{EqMetricOnAutTM}
    \tilde G^{\operatorname{Aut}(TM)}_{(\ph,\al)}((\delta \ph, \delta \alpha),(\delta \ph, \delta \alpha)) \coloneqq \int g_0(v,v) +\tfrac{d\const}{4}\operatorname{Tr}(gA g^{-1}A^\top)\vol(g)\,,
\end{equation}
where $(A,v)$ are given by~\eqref{eq:righttriv}. A straightforward calculation shows that, in Lagragian coordinates, we exactly obtain the metric from Theorem~\ref{thm:diagram}, i.e., 
$\tilde G^{\operatorname{Aut}(TM)}= G^{\operatorname{Aut}(TM)}$. The advantage of writing the metric in Eulerian coordinates is that immediately implies that the action map from $\operatorname{Aut}(TM)$ to the induced Riemannian metric on $\operatorname{Met}(M)$ is a Riemannian submersion and could be thus used to obtain an alternative proof of Theorem~\ref{thm:diagram}.


\end{document}